\documentclass[reqno,12pt,a4wide]{article}

\usepackage{eucal,enumitem,mathrsfs}
\usepackage{amsmath,amssymb,epsfig,bbm,amsthm}

\usepackage[pagewise]{lineno}\nolinenumbers

\numberwithin{equation}{section}

\newtheorem{theorem}{Theorem}[section]

\newtheorem{lemma}[theorem]{Lemma}
\newtheorem{proposition}[theorem]{Proposition}
\newtheorem{definition}[theorem]{Definition}

\newtheorem{example}[theorem]{Example}

\newtheorem{remark}[theorem]{Remark}

\usepackage{hyperref}

\renewcommand{\SS}{\mathscr{S}}
\newcommand{\UU}{\mathscr{U}}
\newcommand{\TT}{\mathscr{T}}
\newcommand{\sfz}{\mathsf z}
\newcommand{\RR}{\mathscr {R}}
\newcommand{\LL}{\mathscr{L}}
\newcommand{\sfx}{\mathsf x}
\newcommand{\sfy}{\mathsf y}
\newcommand{\sft}{\mathsf t}
\newcommand{\rmx}{{\mathrm x}}

\title{Minimal solutions to generalized $\Lambda$-semiflows and gradient flows in metric spaces}
\begin{document}

\author{
Florentine Catharina Flei$\ss$ner 
\thanks{Technische Universit\"at M\"unchen  email:
  \textsf{fleissne@ma.tum.de}.
  } 
}
 \date{}

\maketitle

\begin{abstract}    
Generalized $\Lambda$-semiflows are an abstraction of semiflows with non-periodic solutions, for which there may be more than one solution corresponding to given initial data. A select class of solutions to generalized $\Lambda$-semiflows is introduced. It is proved that such minimal solutions are unique corresponding to given ranges and generate all other solutions by time reparametrization. Special qualities of minimal solutions are shown. 

The concept of minimal solutions is applied to gradient flows in metric spaces and generalized semiflows. Generalized semiflows have been introduced by Ball in \cite{ball2000continuity}.  
\end{abstract}

{\small\tableofcontents}

\section{Introduction} 

Minimal solutions form a particular class of solutions to evolution problems possibly having more than one solution corresponding to given initial data. The idea is to select one particular solution corresponding to each given range of solutions. 

The concept is introduced in [\cite{sf2017}, Section 3] for gradient flows in Hilbert spaces, generated by continuously differentiable functions. In \cite{sf2017}, the reverse approximation of gradient flows as minimizing movements is studied; the notion of minimal solutions proves crucial in the considerations therein.

In the present paper, an abstract approach is taken with the aim of introducing the concept of minimal solutions to a wide variety of evolution problems with non-unique solutions, on a topological space $\SS$, endowed with a Hausdorff topology. 

\medskip\noindent
\paragraph{Minimal solutions} 

A partial order $\succ$ between solutions $u:[0, +\infty) \to \SS$ sharing the same range $\RR = \RR[u] := u([0, +\infty))$ in $\SS$ is defined. We say that $u\succ v$ if there exists an increasing $1$-Lipschitz map $\sfz: [0, +\infty) \to [0, +\infty)$ with $\sfz(0) = 0$ such that
\begin{equation}\label{eq: intro 1}
u(t)  =  v(\sfz(t)) \quad \text{ for all } t\geq 0. 
\end{equation}
A solution $u$ is minimal if for every solution $v$, $u\succ v$ yields $u = v$. (see Definition \ref{def: minimal solution})

Within the abstract framework of generalized $\Lambda$-semiflow (introduced in Section \ref{subsec: 1}), it is shown that, under natural hypotheses, 
\begin{enumerate}
\item there exists a unique minimal solution corresponding to each range $\RR = \RR[u]$, \label{itm: intro 1}
\item each minimal solution induces all other solutions with the same range by time reparametrization (\ref{eq: intro 1}), and \label{itm: intro 2}
\item reaches every point in the range in minimal time \label{itm: intro 3} (see Theorem \ref{thm: minimal solution}).
\end{enumerate}  

\paragraph{Abstraction of semiflow} 

An established basic concept in the study of evolution problems with unique solutions (corresponding to given initial data) is that of a semiflow. 
A semiflow on a metric space $\SS$ is a family of continuous mappings $S(t): \SS\to\SS, \ t\geq 0,$ for which the semigroup properties
\begin{displaymath}
S(0)x = x , \quad S(t+s)x= S(t)S(s)x \quad \quad (x\in\SS, \ s,t\geq0)
\end{displaymath} 
hold; 
$t\mapsto S(t)x$ is identified with the unique solution $u:[0, +\infty) \to \SS$ with initial value $u(0) = x$.
\begin{footnote}
{This definition of semiflow corresponds to the one given in \cite{ball2000continuity} where the continuity of the solutions is not assumed in the definition but regarded as additional assumption.}
\end{footnote} 

Diverse methods are known to abstract dynamical systems, allowing for nonuniqueness of solutions.

One method is to define $S(t)$ as a set-valued mapping and to interpret $S(\cdot)x$ as the collection of all the solutions $u: [0, +\infty) \to \SS$ with initial value $u(0) = x$ (multivalued semiflow, e.g. \cite{babin1986maximal, babin1995attractor, melnik1998attractors}).
Another method is to consider a semiflow $S(\cdot)$ defined on the space of maps $u: [0, +\infty) \to \SS$ (not on the phase space $\SS$), by $S(t)u = u^t$, where $u^t(\cdot):= u(\cdot + t)$ \cite{sell1973differential}.
A third method \cite{ball2000continuity} is to take the solutions themselves as objects of study and generalize the concept of semiflow on the basis that a semiflow $S(\cdot)$ can be equivalently defined as the family of maps $u:[0, +\infty) \to \SS, \ u(t) = S(t)u(0)$. 

\begin{definition}\label{def: generalized semiflow}(J. M. Ball \cite{ball2000continuity}) A \textbf{generalized semiflow} $\UU$ on $\SS$ is a family of maps $u: [0, +\infty) \to \SS$ (called \textbf{solutions}) satisfying the hypotheses 
\begin{enumerate}[label={(G\arabic*)}]
\item  Existence: For each $u_0\in\SS$ there exists at least one $u\in\UU$ with $u(0) = u_0$. \label{itm: G1} 
\item Translates of solutions are solutions: If $u\in\UU$ and $\tau\geq 0$, then the map $u^\tau(t) := u(t+\tau), \ t\in[0, +\infty),$ belongs to $\UU$. \label{itm: G2}
\item Concatenation: If $u, v \in \UU, \ \bar{t}\geq 0$, with $v(0) = u(\bar{t})$, then $w\in\UU$, where $w(t) := u(t)$ for $0\leq t\leq \bar{t}$ and $w(t) := v(t-\bar{t})$ for $t > \bar{t}$.  \label{itm: G3}
\item Upper-semicontinuity with respect to initial data: If $u_j\in\UU$ with $u_j(0) \stackrel{\SS}{\to} x$, then there exist a subsequence $u_{j_k}$ of $u_j$ and $u\in\UU$ with $u(0) = x$ such that $u_{j_k}(t) \stackrel{\SS}{\to} u(t)$ for each $t\geq 0$. \label{itm: G4}
\end{enumerate}
\end{definition}
If in addition the hypothesis \ref{itm: S} is satisfied, then $\UU$ is a \textbf{semiflow}: 
\begin{enumerate}[label = {(S)}]
\item For each $u_0\in\SS$ there is exactly one $u\in\UU$ with $u(0) = u_0$. \label{itm: S}
\end{enumerate}


\medskip
\paragraph{Generalized $\Lambda$-semiflow}

The concept of generalized $\Lambda$-semiflow introduced in this paper is an abstraction of semiflows with non-periodic solutions, where nonuniqueness phenomena may occur (see Section \ref{subsec: 1}).
\begin{footnote} 
{`nonperiodic' means that there is no periodic nonconstant solution 

The $\Lambda$ in `generalized $\Lambda$-semiflow' is no parameter; it reminds of the presence of a Lyapunov or Lyapunov-like function which is a typical example for a situation in which periodic nonconstant solutions are excluded.}
\end{footnote} 

As in \cite{ball2000continuity}, a semiflow is defined as a family of maps $u: [0, +\infty) \to \SS$ satisfying the hypotheses \ref{itm: G1} - \ref{itm: G4} and \ref{itm: S}. 
The solutions themselves are taken as objects of study. However, in the study of minimal solutions, the dynamics between solutions sharing the same range are of interest, rather than the limit behaviour \ref{itm: G4} of solutions possibly having different ranges. The definition of generalized $\Lambda$-semiflow mirrors this aspect.  

A \textbf{generalized $\Lambda$-semiflow} on $\SS$ is defined to be a nonempty family of maps $u: [0, +\infty) \to \SS$ (called \textbf{solutions}) satisfying hypotheses relating to
\begin{enumerate}[label = {($\Lambda$\arabic*)}]
\item time translation: time translates of solutions are solutions, 
\item concatenation: the concatenation of two solutions yield a solution, 
\item non-periodicity: if $u(s) = u(t)$, then $u$ constant in $[s, t]$, 
\item extension: ``if the range of a solution can be passed through in finite time without becoming eventually constant, it may be extended",
\item `local' character: solutions are characterized by their behaviour in finite time intervals
\end{enumerate}
(see Definition \ref{def: sigma semiflow}). We will focus on generalized $\Lambda$-semiflows with sequentially continuous solutions.

\medskip
We note that 
\begin{itemize}
\item there is a connection between the concept of generalized $\Lambda$-semiflow and Ball's concept of generalized semiflow (see next page, \textit{minimal solutions to generalized semiflows});
\item the hypotheses constituting a generalized $\Lambda$-semiflow are mild enough to allow of applications of the theory of minimal solutions to cases beyond the scope of generalized semiflows (see next page, \textit{minimal solutions to gradient flows in metric spaces}).
\end{itemize}

\medskip\noindent
\paragraph{Minimal solutions to generalized semiflows}

Every generalized semiflow with non-periodic continuous solutions is a generalized $\Lambda$-semiflow and all our results \ref{itm: intro 1}, \ref{itm: intro 2}, \ref{itm: intro 3} relating to existence, uniqueness and characteristics of minimal solutions are applicable (see Theorem \ref{thm: ms gs} and Theorem \ref{thm: 3 neu}).

\medskip\noindent
\paragraph{Minimal solutions to gradient flows in metric spaces}

A gradient flow on a metric space $\SS$ \cite{AmbrosioGigliSavare05}, generated by a functional $\phi: \SS\to (-\infty, +\infty]$ and its strong upper gradient $g: \SS \to [0, +\infty]$, is described by the energy dissipation inequality
\begin{displaymath}
\phi(u(s)) - \phi(u(t)) \ \geq \ \frac{1}{2}\int_s^t{g^2(u(r)) \ dr} \ + \ \frac{1}{2}\int_s^t{|u'|^2(r) \ dr}
\end{displaymath} 
for all $0\leq s\leq t$; the solutions $u: [0, +\infty) \to \SS$ are referred to as curves of maximal slope for $\phi$ w.r.t. $g$ (see definitions in Section \ref{subsec: gf1}). 

\medskip
If $\phi$ and $g$ are lower semicontinuous and $\phi$ is quadratically bounded from below, then the corresponding gradient flow is a generalized $\Lambda$-semiflow and all our results \ref{itm: intro 1}, \ref{itm: intro 2}, \ref{itm: intro 3} relating to existence, uniqueness and characteristics of minimal solutions are applicable (see Theorem \ref{thm: gf as gls} and Theorem \ref{thm: 59}). 

\medskip
It is true that our assumptions do not suffice to guarantee a priori the existence of curves of maximal slope but if solutions exist, our concept of minimal solutions can be applied. 

\medskip
Further, a special quality of minimal solutions to a gradient flow can be proved: a curve of maximal slope is a minimal solution if and only if it crosses the $0$ level set of the strong upper gradient $g$ in an $\LL^1$-negligible set of times (before it possibly becomes eventually constant) (see Proposition \ref{prop: cross critical set}).  

\medskip
We note that, under our assumptions, the gradient flow is a generalized $\Lambda$-semiflow but does not fit into the concept of generalized semiflow; additional assumptions such as the relative compactness of the sublevels of $\phi$ (which entails that $\phi$ is bounded from below by a constant) and a conditional continuity assumption would be needed in order to prove the upper-semicontinuity hypothesis \ref{itm: G4} in the Definition \ref{def: generalized semiflow} of generalized semiflow (cf. \cite{rossi2011global} where the theory of generalized semiflow \cite{ball2000continuity} is used to prove the existence of the global attractor for a gradient flow). 

\medskip\noindent
\paragraph{Further results}

If there exists a function $\Psi: \SS\to\mathbb{R}$ which decreases along solution curves, a characterization of minimal solutions in terms of $\Psi$ is also provided (see Proposition \ref{prop: 1}). Time translation and concatenation of minimal solutions yield minimal solutions (see Proposition \ref{prop: ms are gls}). 

\medskip\noindent
\paragraph{Plan of the paper} 

In Section \ref{sec: 1}, we give the precise definitions of generalized $\Lambda$-semiflow, explaining our hypotheses and the link to the classical notion of semiflow, and of minimal solutions, and we prove results relating to existence, uniqueness and characteristics of minimal solutions, within the abstract framework of generalized $\Lambda$-semiflow. 
In Sections \ref{sec: generalized semiflows} and \ref{sec: gf}, we apply our concept of minimal solutions to generalized semiflows (Section \ref{sec: generalized semiflows}) and to gradient flows in metric spaces (Section \ref{sec: gf}).

\section{Notation}

The phase space $\SS$ is endowed with a Hausdorff topology and $x_j\stackrel{\SS}{\to} x$ denotes the corresponding convergence of sequences. 

\medskip\noindent
The range of a curve $u: [0, +\infty) \to \SS$ is denoted by 
\begin{displaymath}
\RR[u]:= u([0, +\infty)), 
\end{displaymath}
its union with what is usually referred to as $\omega$-limit set in the literature by 
\begin{displaymath}
\overline{\RR[u]} := \RR[u]\cup\{w_\star\in\UU \ | \ \exists t_n \to +\infty, \ u(t_n) \stackrel{\SS}{\to} w_\star\}, 
\end{displaymath}  
and we set 
\begin{displaymath}
T_\star(u) := \inf \{s\geq 0 \ | \ u(t) = u(s) \text{ for all } t\geq s\} \in [0, +\infty].
\end{displaymath}

We say that the limit $\lim_{t\uparrow \nu}u(t)=:w_\star\in\SS$ exists for $\nu\in(0, +\infty]$ iff $u(t_n)\stackrel{\SS}{\to}w_\star$ for every sequence of times $t_n\uparrow \nu$. 

\section{Generalized $\Lambda$-semiflow, minimal solutions}\label{sec: 1}
We develop an abstract framework for our analysis of evolution problems for which there may be more than one solution sharing the same range. In this context, we define \textit{generalized $\Lambda$-semiflows}, generalizing the notion of semiflows with Lyapunov function to a certain extent adapted for our considerations.   

\subsection{Definition of generalized $\Lambda$-semiflow}\label{subsec: 1}

\begin{definition}\label{def: sigma semiflow} 
A \textbf{generalized $\Lambda$-semiflow} $\UU$ on $\SS$ is a nonempty family of maps $u: [0, +\infty) \to \SS$  satisfying the hypotheses:
\begin{enumerate}[label={(H\arabic*)}]
\item For every $u\in\UU$ and $\tau\geq 0$, the map $u^{\tau}(t) := u(t + \tau), \ t\in[0, +\infty)$, belongs to $\UU$.\label{itm: 1}
\item Whenever $u, v \in \UU$ with $v(0) = u(\bar{t})$ for some $\bar{t} \geq 0$, then the map $w: [0, +\infty) \to \SS$, defined by $w(t):=u(t)$ if $t\leq\bar{t}$ and $w(t):=v(t-\bar{t})$ if $t > \bar{t}$, belongs to $\UU$. \label{itm: 2}
\item Whenever $u, v \in\UU$ with $v([s, t])\subset \RR[u]$ for some $t > s \geq 0$, then for every $l_1, l_2 \in [s,t]$ the following holds: if $v(l_1)=u(r_1)$ and $v(l_2)=u(r_2)$ with $u(r_1)\neq u(r_2)$ and $r_1 < r_2$, then $l_1 < l_2$. \label{itm: 3}
\item If $u\in\UU$ and there exists a map $w: [0, \theta) \to \SS$ with $\theta < +\infty$ such that $w|_{[0, T]}$ can be extended to a map in $\UU$ for every $T\in[0,\theta)$, and $w([0, \theta)) = \RR[u]$, 
then the limit $\lim_{t\uparrow +\infty}u(t)=:w_\star \in \SS$ exists 
and the map $\bar{w}: [0, +\infty) \to \SS$, defined by 
\begin{displaymath}
\bar{w}(t) := 
\begin{cases}
 w(t) &\text{if } t < \theta \\ w_\star &\text{if } t\geq \theta 
\end{cases}
\end{displaymath} 
belongs to $\UU$. \label{itm: 4} 
\item If a map $w: [0, +\infty) \to \SS$ has the property that $w|_{[0, T]}$ can be extended to a map in $\UU$ for every $T >0$, then $w\in\UU$. \label{itm: 5} 
\end{enumerate} 
The elements $u\in\UU$ are referred to as \textbf{solutions}.  
\end{definition}
The hypotheses \ref{itm: 1} and \ref{itm: 2} say that time translates of solutions are solutions and that the concatenation of two solutions yield a solution. It appears that both axioms arise quite naturally in generalizations of semiflow theory including nonuniqueness phenomena (cf. \cite{ball2000continuity} and Definition \ref{def: generalized semiflow}). 

The meaning of hypothesis \ref{itm: 3} is that there is only one proper direction to run through the range of a solution. Typical examples (as given in this paper) are situations involving an energy decreasing along solution curves and which is constant along a solution only if the solution is constant. As a consequence of \ref{itm: 3} (by choosing $u=v$) we also obtain
\begin{equation}\label{eq: 3}
u(s) = u(t) \quad \text{if and only if} \quad u(r) = u(s)  \text{ for all } r\in [s, t]
\end{equation}    
for all $u\in\UU$ and $0\leq s < t < +\infty$. 
\begin{remark}\label{rem: 3}
Hypothesis \ref{itm: 3} may be replaced by (\ref{eq: 3}) in Definition \ref{def: sigma semiflow}. 

Indeed, if the translation and concatenation hypotheses \ref{itm: 1} and \ref{itm: 2} hold good and all $u\in\UU$ satisfy (\ref{eq: 3}), then \ref{itm: 3} follows by a contradiction argument: suppose that there exist $u, v \in \UU$ and $r_1 < r_2, \ l_2 < l_1$ such that $v(l_1) = u(r_1) \neq u(r_2) = v(l_2)$, and construct the map $w: [0, +\infty) \to \SS$, 
\begin{displaymath}
w(t) := 
\begin{cases}
u(t) &\text{ if } t\leq r_2 \\
v(t + l_2 - r_2) &\text{ if } t > r_2 
\end{cases}
\end{displaymath}
which belongs to $\UU$ by \ref{itm: 1} and \ref{itm: 2}. Then $w(r_1) = w(r_2 + l_1 - l_2)$, but $w(r_2) \neq w(r_1)$ and $r_1 < r_2 < r_2 + l_1 - l_2$, in contradiction to (\ref{eq: 3}).  
 
Conversely, \ref{itm: 3} implies (\ref{eq: 3}), as already mentioned.   
\end{remark}

The extension property expressed in hypothesis \ref{itm: 4} excludes degenerate cases corresponding to the rate at which the range of a solution is described. We give an example of such degenerate case which should be excluded. 
\begin{example}
Let $\SS = \mathbb{R}$ and $\UU$ be the family of all continuous maps $u: [0, +\infty) \to \mathbb{R}$ satisfying $u(0) > 0$ and 
\begin{displaymath}
u'(t) = u(t)^2 \text{ if } t\in(S_i, T_i), \ i\in\mathbb{N} , \quad u'(t) = u(t) \text{ if } t\notin \bigcup_{i\in\mathbb{N}}^{}{[S_i, T_i]}
\end{displaymath} 
for some 
$S_{i+1} \geq T_i \geq S_i \geq 0$ with $\{S_i, \ T_i \ | \ i\in\mathbb{N}\}\cap [0, T]$ finite set for every $T > 0$. Then obviously $\UU$ is nonempty and the hypotheses \ref{itm: 1} - \ref{itm: 3} and \ref{itm: 5} hold good but choosing $w: [0, 1) \to \mathbb{R}, \ w(t) := \frac{1}{1-t}$, we see that $\UU$ does not satisfy \ref{itm: 4}.  
\end{example}

Hypothesis \ref{itm: 5} reflects the `local character' of $\UU$. The following example provides a classic case of a non-local characterization being tantamount to some arbitrariness which we intend to exclude by hypothesis \ref{itm: 5}.  
\begin{example}
Let $\SS = \mathbb{R}^2$ and $\UU$ be the family of all continuous maps $u: [0, +\infty) \to \mathbb{R}^2, \ u(t) = (u_1(t), u_2(t))$ such that $u_1(0) > 0$, $u_2$ is strictly increasing and 
\begin{displaymath}
u_1'(t) = u_1(t) \text{ for all } t > 0, \quad \exists T \geq 0: \ u_2(t) = u_2(T) + t - T \text{ for all } t > T. 
\end{displaymath}
Then it is easy to check that $\UU$ is nonempty and satisfies \ref{itm: 1} - \ref{itm: 4} but $\UU$ does not satisfy \ref{itm: 5}. In this case, any strictly increasing continuous map $g: [0, +\infty) \to \mathbb{R}$ which does not eventually become linear will yield a counterexample to \ref{itm: 5}. 
\end{example}

Let us explain to what extent our notion of generalized $\Lambda$-semiflow is an abstraction of the classical semiflow theory.

We observe that any semiflow $\UU$ whose members satisfy (\ref{eq: 3}) is a generalized $\Lambda$-semiflow. This follows from the time translation and uniqueness property (corresponding to given initial data) of a semiflow (\ref{itm: G2} and \ref{itm: S}). It is straightforward to check \ref{itm: 1} - \ref{itm: 3} in this case. Choosing $u\in\UU$ and $w: [0, \theta) \to \SS$ as in \ref{itm: 4}, we obtain $u|_{[0, \theta)} = w|_{[0, \theta)}$ by \ref{itm: S} (since $u(0) = w(0)$ by \ref{itm: 3}) so that $w([0, \theta)) = \RR[u]$ and (\ref{eq: 3}) yield $u$ constant in $[\theta, +\infty)$. This proves \ref{itm: 4}. Finally, \ref{itm: 5} follows from \ref{itm: S}.   

On the other hand, if a member $u: [0, +\infty) \to \SS$ of a semiflow does not satisfy (\ref{eq: 3}), then there is necessarily a time $T > 0$ such that $u$ is periodic and nonconstant on $[T, +\infty)$. Indeed, if there exist $0\leq s < \bar{r} < t < +\infty$ such that $u(s) = u(t)$ but $u(\bar{r})\neq u(s)$, then \ref{itm: G2} and \ref{itm: S} imply $u(r + s) = u(r + t)$ for all $r\geq 0$ which is equivalent to 
\begin{displaymath}
u(r + t - s) = u(r) \text{ for all } r\geq s, \quad u(\bar{r} + j(t-s))\neq u(s + j(t-s)) \text{ for all } j\in\mathbb{N}. 
\end{displaymath}
The hypotheses \ref{itm: 3} and \ref{itm: 4} do not hold good in this case. We illustrate this situation excluded in Definition \ref{def: sigma semiflow} with an example.   
\begin{example}
Let $\SS = \mathbb{R}^2$ and consider 
\begin{displaymath}
\UU := \{u: [0, +\infty) \to \mathbb{R}^2 \ | \ u(\cdot) \equiv r(\cos(\cdot + \tau), \ \sin(\cdot + \tau)),  \quad \tau\in [0, 2\pi), \ r\geq 0 \}.
\end{displaymath}
Clearly, $\UU$ is a semiflow on $\mathbb{R}^2$ but the hypotheses \ref{itm: 3} and \ref{itm: 4} are not satisfied and $\UU$ is not a generalized $\Lambda$-semiflow. 
\end{example}  

A connection between generalized $\Lambda$-semiflows and the established theory of generalized semiflows introduced by Ball \cite{ball2000continuity} is made in Section \ref{sec: generalized semiflows}. We will see that any generalized semiflow whose members satisfy (\ref{eq: 3}) satisfies the hypotheses \ref{itm: 1} - \ref{itm: 3}, \ref{itm: 5} and a slightly weaker variation on \ref{itm: 4}. 
If, in addition, all the solutions are continuous, then it satisfies \ref{itm: 4}, too. 

We note that a generalized $\Lambda$-semiflow $\UU$ on $\SS$ is nonempty but there may be initial data $x\in\SS$ for which there exists no $u\in\UU$ with $u(0) = z$. Also, nothing is said about the behaviour of a sequence $(u_j)$ in $\UU$ with converging initial data $u_j(0)$.

Gradient flows in metric spaces fit very well in the concept of generalized $\Lambda$-semiflows. This aspect is examined in Section \ref{sec: gf}.

\subsection{A partial order between solutions}

Let a generalized $\Lambda$-semiflow $\UU$ on $\SS$ be given. We introduce a particular class of solutions (which we call \textit{minimal solutions}), arising naturally from a partial order in $\UU$:  

\begin{definition}\label{def: minimal solution}
  If $u,v\in \UU$ we say that 
  $u\succ v$ if $\RR[v]\subset  \overline{\RR[u]}
 $ 
 and there exists an increasing $1$-Lipschitz map $\sfz
 :[0,+\infty)\to [0,+\infty)$ with $\sfz(0)=0$ such that
 \begin{equation}
   \label{eq: partial order}
  u(t)=v(\sfz(t))\quad\text{for every }t\ge 0.
\end{equation}
An element $u\in \UU$ is \textbf{minimal} if for every $v\in \UU $,  $u\succ v$ yields $u=v$; and $\UU_{\text{min}}$ denotes the
collection of all the minimal solutions.
\end{definition}

Let us make a few comments on Definition \ref{def: minimal solution}. 
\begin{enumerate}[label={ ({\roman*}) }]
\item A map $\sfz: [0, +\infty) \to [0, +\infty)$ is increasing and $1$-Lipschitz if and only if 
\begin{equation}
0\le \sfz(t)-\sfz(s)\le t-s \quad \text{for every }0\le s\le t.
\end{equation}
\item It is not difficult to see that $\succ$ forms indeed a partial order in $\UU$ ([\cite{sf2017}, Remark 3.3]).
\item Condition (\ref{eq: partial order}) implies the range inclusion $\RR[u]\subset \RR[v]$.
\item The condition on the range $\RR[v]\subset \overline{\RR[u]}$ gives control over the long-time behaviour of a possible minimal solution. Its effect as a selection criterion is illustrated in [\cite{sf2017}, Remark 3.2] with a 1-dimensional example of a gradient flow.    
\end{enumerate}

It is not clear a priori if minimal solutions exist at all. Some kind of compactness property of $\UU$ appears necessary in order to guarantee the existence of minimal solutions. Let us consider our main tools concerning compactness for the existence proof given in section \ref{sec: existence}. \\

We introduce the class of truncated solutions 
\begin{displaymath}
\TT[\UU] := \{v: [0, +\infty) \to \SS \ | \ v(t) = u(t\wedge T) \quad \text{for some } u\in\UU, \ T\in [0, +\infty]\}
\end{displaymath}
and we define the map $\rho: \TT[\UU] \to [0, +\infty]$ as 
\begin{equation}\label{eq: rho}
\rho(v) := \inf\{s\geq 0 \ | \ v(t) = v(s) \quad \text{for every } t\geq s\}, \ v\in\TT[\UU].  
\end{equation} 

The following \textbf{compactness hypothesis} will turn out to be appropriate for our purposes:   
\begin{enumerate}[label = {(C)}]
\item If a sequence $v_n\in\TT[\UU], \ n\in\mathbb{N},$ satisfies $\sup_n{\rho(v_n)} < +\infty$ and $\RR[v_n] = \RR[v_1]$ for all $n \in\mathbb{N}$, then there exists $v\in\TT[\UU]$ and a subsequence $n_k\uparrow +\infty$ such that 
\begin{displaymath}
v_{n_k}(t)\stackrel{\SS}{\to} v(t) \quad \text{for all } t\in[0, +\infty), \quad \RR[v] = \RR[v_1].
\end{displaymath} \label{itm: C}
\end{enumerate} 
We note that in the above situation it holds that
\begin{equation}\label{eq: rho lsc}
\rho(v) \ \leq \ \liminf_{k\to +\infty} \rho(v_{n_k})
\end{equation}
since $\rho$ is lower semicontinuous with respect to pointwise convergence. \\

Now, we have all the ingredients to prove the existence of minimal solutions. Our construction will be based on a step-by-step procedure of truncating a given trajectory and each time minimizing $\rho$ with respect to the truncated range.  

\subsection{Existence and characteristics of minimal solutions}\label{sec: existence}

Existence and uniqueness of minimal solutions corresponding to given ranges is proved under the additional compactness hypothesis \ref{itm: C}.

It is shown that among solutions sharing the same range, the minimal solution induces all the other ones by time reparametrization (\ref{eq: partial order}) and it reaches any point in the range in minimal time. 

\medskip\noindent
\textbf{Definition of $\UU[\RR]$}

For a generalized $\Lambda$-semiflow $\UU$ and the range $\RR = \RR[y] \subset \SS$ of a solution $y\in\UU$, we define $\UU[\RR]$ as the collection of all the solutions $w\in\UU$ with $\RR\subset\RR[w]\subset \overline{\RR} := \RR\cup\{w_\star\in\SS \ | \ \exists t_n\to +\infty, \ y(t_n)\stackrel{\SS}{\to} w_\star\}$ and 
\begin{equation}\label{eq: U[R]}
w([0, \theta)) = \RR \quad \text{and} \quad  w([\theta, +\infty)) \subset \overline{\RR}\setminus \RR \quad \text{for some } \theta \in (0, +\infty]. 
\end{equation}
We note that the set $\overline{\RR}$ is indeed independent of the choice $y\in\UU$ with $\RR[y] = \RR$: 
\begin{lemma}
Whenever $y, \tilde{y}\in\UU, \ \RR[y] = \RR[\tilde{y}]$ and $w_\star\in\SS$, it holds that
\begin{equation}\label{eq: overline R}
\exists t_n\to +\infty, \ y(t_n)\stackrel{\SS}{\to} w_\star \quad \text{if and only if} \quad \exists s_n\to +\infty, \ \tilde{y}(s_n)\stackrel{\SS}{\to} w_\star, 
\end{equation}
i.e. it holds that $\overline{\RR[y]} = \overline \RR = \overline{\RR[\tilde{y}]}$.
\end{lemma} 
\begin{proof}
If $t_n\to +\infty, \ y(t_n) \stackrel{\SS}{\to} w_\star$, then there is a sequence of times $(s_n)$ with $\tilde{y}(s_n) = y(t_n)$, and by \ref{itm: 3}, we may assume that $(s_n)$ is increasing. Let $S:= \sup_n s_n$. If $S = +\infty$ or $T_\star(y) < +\infty$, nothing remains to be shown. If $S < +\infty$ and $T_\star(y)=+\infty$, then we obtain $\tilde{y}([0, S)) = \RR[y] = \RR[\tilde{y}]$ by \ref{itm: 3}, and thus by (\ref{eq: 3}) there exists $\delta > 0$ such that $\tilde{y}$ is constant in $(S-\delta, +\infty]$, in contradiction to $T_\star(y) = +\infty$. This proves (\ref{eq: overline R}).    
\end{proof}

Let us take a close look at the case of finite $\theta$ in (\ref{eq: U[R]}). If there exists a solution $w\in\UU[\RR]$ with $w([0, \theta)) = \RR$ and $\theta < +\infty$, we may apply \ref{itm: 4} and obtain that the limit $\lim_{t\uparrow +\infty} y(t) =: w_\star \in \SS$ is well-defined and that $w(t) = w_\star$ for all $t\geq \theta$. We notice that $\overline \RR$ then takes the form $\overline \RR = \RR\cup\{w_\star\}$. In this case, $\overline \RR = \RR[w] = \overline{\overline \RR}$ and $\UU[\overline \RR] \subset \UU[\RR]$.  

The following observation which is a direct consequence of Definition \ref{def: minimal solution} and (\ref{eq: 3}) may be seen as motivation behind considering $\UU[\RR]$.
\begin{lemma}
For $y, w\in\UU$, the implication 
\begin{equation}\label{eq: U[R] motiv}
y\succ w \quad \Rightarrow \quad w\in\UU[\RR[y]] 
\end{equation}
holds good. 
\end{lemma}
\begin{proof}
If $y\succ w$, then by definition, $\RR[w]\subset \overline{\RR[y]}$ and there exists an increasing $1$-Lipschitz map $\sfz: [0, +\infty) \to [0, +\infty)$ with $\sfz(0) = 0$ such that $y(t) = w(\sfz(t))$ for all $t\geq 0$. Choose $\theta := \sup_{t\geq0} \sfz(t) \in [0, +\infty]$. If $s\in[0, \theta)$, then there exists $t\geq0$ such that $\sfz(t) = s$ and thus $w(s) \in \RR[y]$. If $\theta = +\infty$, then $\RR[w] = \RR[y]$. The same holds if $\theta < +\infty, \ \sfz(\bar{t}) = \theta$ for some $\bar{t}\geq0$. Finally, we consider the case $\theta < +\infty, \ \sfz(t) < \theta$ for all $t\geq0$. It holds that $w([0, \theta)) = \RR[y]$ and $w([\theta, +\infty) \subset \overline{\RR[y]}$. If $w(s) \in \RR[y]$ for some $s\geq \theta$, then there exists $\tilde{s}\in[0, \theta)$ such that $w(s) = w(\tilde{s})$, and by (\ref{eq: 3}), $w$ is constant in $[\tilde{s}, s]$, hence $T_\star(y) < +\infty$ and $\RR[w] = \RR[y]$. The proof of (\ref{eq: U[R] motiv}) is complete.     
\end{proof}  

Now, our theorem reads as follows. 

\begin{theorem}\label{thm: minimal solution}
Let $\UU$ be a generalized $\Lambda$-semiflow on $\SS$ satisfying the compactness hypothesis \ref{itm: C}. Suppose that every solution $u\in\UU$ is sequentially continuous, i.e.
\begin{equation}\label{eq: cont h}
u(t_j) \stackrel{\SS}{\to} u(t) \quad \text{whenever } t_j \to t, \ t_j, t \in [0, +\infty).
\end{equation}  
Then the following statements hold good: 
\begin{enumerate}[label = {(\arabic*)}]
\item For every $\RR = \RR[y] \subset \SS$ which is the range of a solution $y\in\UU$ there exists a unique minimal solution $u\in\UU[\RR]\cap\UU_{\text{min}}$.

Moreover, if $v\in\UU[\RR]$, then $v\succ u$. \label{itm: T1} 
\item Every minimal solution $u\in\UU_{\text{min}}$ is injective in $[0, T_\star(u))$.\label{itm: T2} 
\item Whenever $u\in\UU_{\text{min}}, \ v\in\UU$ with $u\in\UU[\RR[v]]$ and $u(t_0) = v(t_1)$ for some $t_0, t_1\in[0, +\infty)$, then $t_0\wedge T_\star(u) \leq t_1$.  \label{itm: T3} 
\item Whenever $u\in\UU_{\text{min}}, \ v\in\UU$ with $v([s_1, t_1]) = u([s_0, t_0])$ for some $t_i\geq s_i\geq 0 \ (i=0,1)$, then the inequality 
\begin{displaymath}
t_0\wedge T_\star(u) - s_0 \ \leq \ t_1 - s_1
\end{displaymath}
necessarily holds. \label{itm: T4} 
\item A solution $u\in\UU$ belongs to $\UU_{\text{min}}$ if for every $v\in\UU[\RR[u]]$ the following implication holds: whenever $u(t_0) = v(t_1)$ for some $t_0, t_1 \in [0, +\infty)$, then $t_0\wedge T_\star(u) \leq t_1$. \label{itm: T5}
\end{enumerate}
\end{theorem} 

\begin{proof}
(1). Let $\RR = \RR[y] \subset \SS$ be the range of a solution $y\in\UU$. 
We distinguish between two cases: $T_\star(y) = +\infty$ and $T_\star(y) < +\infty$. In the first case we select an increasing sequence of times $T_n \uparrow +\infty$ with $y(T_n) \neq y(T_{n+1})$ for all $n\in\mathbb{N}$. Then we have
\begin{displaymath}
y([0, T_n]) \subsetneq y([0, T_{n+1}]), \quad \bigcup_{n}^{}{y([0, T_n])} = \RR. 
\end{displaymath} 
If $T_\star(y) < +\infty$, we may go through the following proof with just one step $n=1$ and $T_1 := T_\star(y)$.  

For every $n \ (n\in\mathbb{N} \text{ or } n=1)$, we minimize $\rho$ (defined in (\ref{eq: rho})) in
\begin{displaymath}
\mathcal{G}[\RR_n] := \{w\in\TT[\UU] \ | \ \RR[w] = \RR_n\}, \quad \RR_n := y([0, T_n]).  
\end{displaymath}
Since $y(\cdot \wedge T_n) \in \mathcal{G}[\RR_n]$ and thus $\inf_{w\in\mathcal{G}[\RR_n]} \rho(w) \leq T_n < +\infty$, the compactness hypothesis \ref{itm: C} and (\ref{eq: rho lsc}) yield the existence of a minimizer $u_n\in\mathcal{G}[\RR_n]$ of $\rho|_{\mathcal{G}[\RR_n]}$. By (\ref{eq: cont h}), $u_n$ is constant in $[\rho(u_n), +\infty)$. We show that $u_n$ is the unique minimizer of $\rho$ in $\mathcal{G}[\RR_n]$. Suppose that there exist $\tilde{u}_n\in\mathcal{G}[\RR_n], \ t_0 \geq 0$ with $\rho(\tilde{u}_n) = \rho(u_n), \ \tilde{u}_n(t_0)\neq u_n(t_0)$. Then it follows from \ref{itm: 3} that $t_0\in(0, \rho(u_n))$ and that there exists $s_0\in [0, \rho(u_n))$, w.l.o.g. $s_0 < t_0$, such that $\tilde{u}_n(s_0) = u_n(t_0)$ and $\tilde{u}_n([0, s_0]) = u_n([0, t_0])$. By \ref{itm: 1} and \ref{itm: 2}, we may construct a truncated solution $w\in\mathcal{G}[\RR_n]$,  
\begin{displaymath}
w(r):= \begin{cases}
\tilde{u}_n(r) &\text{ if } r\in[0, s_0] \\
u_n(r+t_0-s_0) &\text{ if } r > s_0
\end{cases}
\end{displaymath} 
satisfying $\rho(w) \leq \rho(u_n) + s_0 - t_0 < \rho(u_n)$, in contradiction to $u_n$ minimizing $\rho$ in $\mathcal{G}[\RR_n]$. So $\rho$ admits a unique minimizer $u_n$ in $\mathcal{G}[\RR_n]$.

The same argument shows that $u_n$ is injective in $[0, \rho(u_n)]$. \\

We now set $S_n := \rho(u_n) \leq T_n$ and define $\sfz_n: [0, T_n] \to [0, S_n]$ as  
\begin{displaymath}
  \sfz_n(t):=\min\Big\{s\in [0,S_n]:u_n(s)=y(t)\Big\}, \quad t\in[0, T_n].  
\end{displaymath}
The map $\sfz_n$ is increasing by \ref{itm: 3}, and $\sfz_n(0) = 0, \ \sfz_n(T_n) = S_n$. It holds that $u_n(\sfz_n(t)) = y(t)$ for all $t\in [0, T_n]$. A contradiction argument shows that $\sfz_n$ is $1$-Lipschitz. Suppose that there exist $t_1, t_2 \in [0, T_n], \ t_1 < t_2$, such that $\delta_t := t_2 - t_1 < \sfz_n(t_2) - \sfz_n(t_1) =: \delta_\sfz$. Then let us construct the map $w: [0, +\infty) \to \SS$, 
\begin{displaymath}
  w(r):=
  \begin{cases}
    u_n(r)&\text{if }0\le r\le \sfz_n(t_1),\\
    y(r+t_1-\sfz_n(t_1))&\text{if }\sfz_n(t_1)\le r\le
    \delta_t+\sfz_n(t_1)\\
    u_n(r+\delta_\sfz-\delta_t)&\text{if }r\ge \delta_t+\sfz_n(t_1).
  \end{cases}
\end{displaymath}
which belongs to $\mathcal{G}[\RR_n]$ by \ref{itm: 1} - \ref{itm: 3}. Moreover, $\rho(w)\leq S_n - \delta_\sfz + \delta_t < S_n$, a contradiction to the fact that $u_n$ minimizes $\rho$ in $\mathcal{G}[\RR_n]$.     \\

A further contradiction argument (which we omit since it is very similar to the preceding two) shows that $S_n < S_{n+1}$ and that $u_n(\cdot \wedge s)$ minimizes $\rho$ in 
\begin{displaymath}
\{w\in\TT[\UU] \ | \ \RR[w] = u_n([0, s])\}
\end{displaymath}
if $s\in[0, S_n]$. In particular, we obtain 
\begin{equation}\label{eq: u_n = u_n+1}
u_n(s) = u_{n+1}(s), \quad \sfz_n(t) = \sfz_{n+1}(t) \quad \text{for every } s\in [0, S_n], \ t\in[0, T_n]. 
\end{equation} 
Let $S_\star := \sup_n S_n$. Due to (\ref{eq: u_n = u_n+1}), we may define $u: [0, S_\star) \to \SS$ as
\begin{equation}\label{eq: u}
u(s) := u_n(s) \quad \text{if } s\in [0, S_n],
\end{equation}
and $\sfz: [0, T_\star(y)) \to [0, S_\star)$ as
\begin{equation}\label{eq: sfz}
\sfz(t) := \sfz_n(t) \quad \text{if } t\in [0, T_n]. 
\end{equation}
If $S_\star = +\infty$, then the map $u: [0, +\infty) \to \SS$ belongs to $\UU$ by \ref{itm: 5}. Since it holds that 
$T_\star(y) = +\infty$ in this case, we obtain $y(t) = u(\sfz(t))$ for all $t\geq 0$. In particular, $y\succ u$. If $S_\star < + \infty$ and $T_\star(y) = +\infty$, we apply hypothesis \ref{itm: 4} which provides that the limit $\lim_{t\uparrow +\infty} y(t) =: u_\star \in \SS$ is well-defined in this case and that extending $u$ by the constant value $u_\star$ yields a map in $\UU$, i.e. $u: [0, +\infty) \to \SS$ defined as   
\begin{equation}\label{eq: u ext}
u(s) := \begin{cases}
u_n(s) &\text{ if } s\in [0, S_n] \\
u_\star &\text{ if } s\geq S_\star
\end{cases}
\end{equation}
belongs to $\UU$. Again we obtain $y(t) = u(\sfz(t))$ for all $t\geq 0$, and thus $y\succ u$. The same goes for the case $S_\star = S_1, \ T_\star(y) < +\infty$: in this case we may extend $u$ as in (\ref{eq: u ext}) due to \ref{itm: 1} and \ref{itm: 2}, and extending $\sfz$ by the constant value $S_\star$, we obtain $y\succ u$.    

We note that $u\in\UU[\RR]$ and $\UU[\RR[u]] \subset \UU[\RR]$. \\

Suppose now that 
\begin{equation}\label{eq: T1}
v\succ u \quad \text{for all } v\in\UU[\RR]. 
\end{equation}
Then, due to (\ref{eq: U[R] motiv}), it follows that for every $\bar{u}\in \UU$, $u \succ \bar{u}$ yields $u = \bar{u}$. This shows that $u\in\UU_{\text{min}}$, and by (\ref{eq: T1}) again, $u$ is the unique minimal solution in $\UU[\RR]$. \\

So it only remains to prove (\ref{eq: T1}): \\

Let $v\in\UU[\RR]$. Let $S_n$ be as in the construction of $u$. For every $S_n$, choose $0\leq \tilde{T}_n \leq T_\star(v)$ such that $v(\tilde{T}_n) = u(S_n)$. By \ref{itm: 3}, $v([0, \tilde{T}_n]) = u([0, S_n])$ and $(\tilde{T}_n)$ is increasing. We set $\tilde{T}_\star := \sup_n \tilde{T}_n\leq T_\star(v)$. We show that $\tilde{T}_\star = T_\star(v)$. Suppose that $\tilde{T}_\star < T_\star (v)$ (which implies $S_\star \leq \tilde{T}_\star < +\infty$). Then we obtain by \ref{itm: 3}, since $v([0, \tilde{T}_\star)) = u([0, S_\star))$, that there exists $\delta > 0$ such that $v$ is constant in $(\tilde{T}_\star - \delta, \tilde{T}_\star)$, contradicting the fact that $v(T_n) \neq v(T_m)$ for $n\neq m$. If there is only one step $n=1$ in the construction of $u$ and $S_\star = S_1$, then we clearly have $T_\star(v) < +\infty$ and $\tilde{T}_\star = \tilde{T}_1 = T_\star(v)$. 

We define $\tilde{\sfz}: [0, T_\star(v)) \to [0, S_\star)$ as
\begin{displaymath}
\tilde{\sfz}(t) := \min \Big\{s\in [0,S_\star):u(s)=v(t)\Big\}, \quad t\in[0, T_\star(v)).
\end{displaymath}
It holds that $v(t) = u(\tilde{\sfz}(t))$ for all $t\in [0, T_\star(v))$.
Following the same arguments as above for $\sfz$, we obtain that $\tilde{\sfz}$ is increasing and $1$-Lipschitz. Extending $\tilde{\sfz}$ by the constant value $S_\star \leq T_\star(v)$ if $T_\star(v) < +\infty$, we obtain $v\succ u$.

The proof of \ref{itm: T1} is complete. 

\medskip\noindent
The statements \ref{itm: T2} - \ref{itm: T5} are direct consequences of our method of constructing the minimal solutions. However, we provide independent proofs. 

\medskip\noindent
(2). Let $u\in\UU_{\text{min}}$ and suppose that there exist $0\leq t_0 < t_1 < T_\star(u)$ such that $u(t_0) = u(t_1)$. By (\ref{eq: 3}) it follows that $u(r) = u(t_0)$ for all $r\in [t_0, t_1]$. Now we define $w: [0, +\infty) \to \SS$ as
\begin{displaymath}
w(t) := \begin{cases}
u(t) &\text{ if } 0\leq t \leq t_0 \\
u(t+t_1-t_0) &\text{ if } t > t_0
\end{cases}
\end{displaymath}
which belongs to $\UU$ by \ref{itm: 1} and \ref{itm: 2}. Choosing $\sfz(t):= t \wedge t_0 + (t-t_1)_+$, we see that $u\succ w$, which yields $w = u$ since $u$ is minimal. This implies $u(r) = u(r + t_1 - t_0)$ for all $r \geq t_0$. Due to (\ref{eq: 3}), it follows that $u$ is constant in $[t_0, +\infty)$, in contradiction to $t_0 < T_\star(u)$. So $u$ is injective in $[0, T_\star(u))$. 

\medskip\noindent
(3) is a special case of \ref{itm: T4}. 

\medskip\noindent
(4). Let $u\in\UU_{\text{min}}, \ v\in\UU$ and $t_i\geq s_i\geq 0$ such that $v([s_1, t_1]) = u([s_0, t_0])$. If $T_\star(u) < +\infty$, we may assume w.l.o.g. that $s_0 < t_0 \leq T_\star(u)$. We note that $v(s_1) = u(s_0)$ and $v(t_1) = u(t_0)$ by \ref{itm: 3}, and define $w: [0, +\infty) \to \SS$ as 
\begin{displaymath}
w(r) := \begin{cases}
u(r) &\text{ if } 0\leq r \leq s_0 \\
v(r+s_1-s_0) &\text{ if } s_0 < r \leq t_1-s_1+s_0 \\
u(r+t_0-s_0+s_1-t_1) &\text{ if } r > t_1-s_1+s_0
\end{cases}
\end{displaymath} 
which belongs to $\UU$ by \ref{itm: 1} and \ref{itm: 2}, with $\RR[w] = \RR[u]$. Due to \ref{itm: T1}, it holds that $w\succ u$, i.e. there exists an increasing $1$-Lipschitz continuous map $\sfz: [0, +\infty) \to [0, +\infty)$ such that $u(\sfz(t)) = w(t)$ for all $t\in [0, +\infty)$. Since $u$ is injective in $[0, T_\star(u))$ (see statement \ref{itm: T2}), it follows that $\sfz(s_0) = s_0$ and $\sfz(t_1 - s_1 + s_0) \geq t_0$. So we obtain
\begin{displaymath}
t_0 - s_0 \ \leq \ \sfz(t_1 - s_1 + s_0) - \sfz(s_0) \ \leq \ t_1 - s_1 
\end{displaymath}  
by the $1$-Lipschitz continuity of $\sfz$. This proves \ref{itm: T4}. 

\medskip\noindent
(5). Suppose that $u\in\UU$ satisfies the assumption of claim \ref{itm: T5} and that $u\succ v$ for some $v\in\UU$. Then there exists an increasing $1$-Lipschitz map $\sfz: [0, +\infty) \to [0, +\infty)$ such that $v(\sfz(t)) = u(t)$ for all $t\in [0, +\infty)$ and $\sfz(0) = 0$, hence $\sfz(t) \leq t$ for all $t\in[0, +\infty)$. Moreover, $v\in\UU[\RR[u]]$ due to (\ref{eq: U[R] motiv}). By assumption of \ref{itm: T5}, it follows that $t\leq \sfz(t)$ for all $t\in [0, T_\star(u))$. Taken together, this yields $\sfz(t) = t$ for all $t\in [0, T_\star(u))$, and thus $u=v$. So we obtain that $u$ is minimal. 

\medskip\noindent
The proof of Theorem \ref{thm: minimal solution} is complete.    
\end{proof}

\begin{remark}\label{rem: C}
In view of Definition \ref{def: minimal solution} and \ref{itm: C}, the sequential continuity (\ref{eq: cont h}) of the solutions appears a natural hypothesis in our concept (cf. the instances under consideration in Sections \ref{sec: generalized semiflows} and \ref{sec: gf}).

We do not make use of the compactness hypothesis \ref{itm: C} and of (\ref{eq: cont h}) in the proof of the statements \ref{itm: T2} and \ref{itm: T5}. 
\end{remark}

Time translates of minimal solutions are minimal soulutions and the concatenation of two minimal solutions yield a minimal solution: 
\begin{proposition}\label{prop: ms are gls}
Let $\UU$ be a generalized $\Lambda$-semiflow on $\SS$. Then it holds:  

For every $u\in\UU_{\text{min}}$ and $\tau\geq 0$, the map $u^{\tau}(t) := u(t + \tau), \ t\in[0, +\infty)$, belongs to $\UU_{\text{min}}$.

Whenever $u, v \in \UU_{\text{min}}$ with $v(0) = u(\bar{t})$ for some $\bar{t} \geq 0$ and $u$ is sequentially continuous (\ref{eq: cont h}), then the map $w: [0, +\infty) \to \SS$, defined by $w(t):=u(t)$ if $t\leq\bar{t}_\star$ and $w(t):=v(t-\bar{t}_\star)$ if $t > \bar{t}_\star$, with $\bar{t}_\star:= \bar{t}\wedge T_\star(u)$, belongs to $\UU_{\text{min}}$.
\end{proposition}

\begin{proof}
We prove the first statement: Let $u\in\UU_{\text{min}}$ and $\tau \geq 0$. Suppose that $u^\tau\succ v$ for some $v\in\UU$. Then $v\in\UU[\RR[u^\tau]]$ and there exists an increasing $1$-Lipschitz map $\sfz: [0, +\infty) \to [0, +\infty)$ such that $u(t+\tau) = u^\tau(t) = v(\sfz(t))$ for all $t\geq 0$. We define $\tilde{v}: [0, +\infty) \to \SS$ as 
\begin{displaymath}
\tilde{v}(t) := \begin{cases}
u(t) &\text{ if } t\leq \tau \\
v(t-\tau) &\text{ if } t>\tau
\end{cases}
\end{displaymath}
which belongs to $\UU$ by \ref{itm: 2}. It holds that $\tilde{v}\in\UU[\RR[u]]$ and choosing $\tilde{\sfz}: [0, +\infty) \to [0, +\infty)$, 
\begin{displaymath}
\tilde{\sfz}(t) := \begin{cases}
t &\text{ if } t\leq \tau \\
\sfz(t-\tau) + \tau &\text{ if } t>\tau
\end{cases}
\end{displaymath}
we obtain $u\succ \tilde{v}$. Since $u$ is minimal, it follows that $u = \tilde{v}$, hence $u^{\tau} = v$ and the claim is proved. 

Now, we prove the second statement: Let $u, v\in\UU_{\text{min}}, \ \bar{t}\geq 0$ be given, set $\bar{t}_\star:=\bar{t}\wedge T_\star(u)$ and define $w:[0, +\infty) \to \SS$ as 
\begin{displaymath}
w(t) := \begin{cases}
u(t) &\text{ if } t\leq \bar{t}_\star \\
v(t-\bar{t}_\star) &\text{ if } t>\bar{t}_\star
\end{cases}
\end{displaymath} 
which belongs to $\UU$ by \ref{itm: 2}. 

Suppose that $w\succ y$ for some $y\in\UU$. Then $y\in\UU[\RR[w]]$ and there exists an increasing $1$-Lipschitz map $\sfz: [0, +\infty) \to [0, +\infty)$ such that $w(t) = y(\sfz(t))$ for all $t\geq 0$. We define $w_i: [0, +\infty) \to \SS \ (i=1,2)$ as
\begin{displaymath}
w_1(t) :=
\begin{cases}
y(t) &\text{ if } t\leq \sfz(\bar{t}_\star) \\
u(t + \bar{t}_\star - \sfz(\bar{t}_\star)) &\text{ if } t > \sfz(\bar{t}_\star)
\end{cases}
 \quad\quad w_2(t) := y(t + \sfz(\bar{t}_\star)). 
\end{displaymath}
Choosing $\sfz_i: [0, +\infty) \to [0, +\infty) \ (i=1,2)$, 
\begin{displaymath}
\sfz_1(t) := \begin{cases}
\sfz(t) &\text{ if } t\leq \bar{t}_\star \\
t+\sfz(\bar{t}_\star) - \bar{t}_\star &\text{ if } t > \bar{t}_\star
\end{cases}
 \quad\quad \sfz_2(t):= \sfz(t + \bar{t}_\star) - \sfz(\bar{t}_\star), 
\end{displaymath}
we see that $u\succ w_1$ and $v\succ w_2$. As $u, v$ are minimal solutions, it follows that $u = w_1, \ v = w_2$. Hence, $y(t) = u(t)$ for all $t \leq \sfz(\bar{t}_\star)$ and $y(t) = v(t-\sfz(\bar{t}_\star))$ for all $t > \sfz(\bar{t}_\star)$. Due to statement \ref{itm: T2} in Theorem \ref{thm: minimal solution}, the minimal solution $u$ is injective in $[0, T_\star(u))$. So, $u=w_1$ implies $\sfz(\bar{t}_\star)= \bar{t}_\star$ and we obtain $y = w$. The proof is complete. 
\end{proof}

\begin{remark}\label{rem: ms are gls}
Clearly, $\UU_{\text{min}}$ satisfies \ref{itm: 3}, and with similar arguments as in the proof of Proposition \ref{prop: ms are gls}, it is possible to show that $\UU_{\text{min}}$ satisfies \ref{itm: 4} and \ref{itm: 5}, too. 

The second statement of Proposition \ref{prop: ms are gls} still holds for $0\leq \bar{t} \leq T_\star(u)$ if we do not assume that $u$ is sequentially continuous. 
\end{remark}

\newpage
\section{Minimal solutions to generalized semiflows}\label{sec: generalized semiflows}

We study the theory developed in Section \ref{sec: 1} with regard to the concept of \textit{generalized semiflows} introduced by Ball \cite{ball2000continuity}.

\medskip
According to \cite{ball2000continuity, ball2004global}, we suppose that $\SS$ is a metric space with metric $d$ and we work with the topology induced by the metric, i.e. 
\begin{displaymath}
x_j \stackrel{\SS}{\to} x \quad :\Leftrightarrow \quad d(x_j, x) \to 0
\end{displaymath}
for $x_j, x\in \SS$.

We refer the reader to Definition \ref{def: generalized semiflow} for the definition of generalized semiflow. For a given generalized semiflow $\UU$, the following is defined in \cite{ball2000continuity}: 

\medskip
A \textit{complete orbit} is a map $w: \mathbb{R}\to\SS$ such that for any $s\in\mathbb{R}$, the map $w^s(t):= w(t+s), \ t\in[0, +\infty),$ belongs to $\UU$. A complete orbit $w$ is \textit{stationary} if $w(t) = x$ for all $t\in\mathbb{R}$, for some $x\in\SS$.
\begin{definition}\label{def: Lyapunov}\cite{ball2000continuity} A function $\psi: \SS\to \mathbb{R}$ is called a \textbf{Lyapunov function} for $\UU$ if the following holds
\begin{enumerate}[label={(L\arabic*)}]
\item $\psi$ is continuous, \label{itm: L1}
\item $\psi(u(t)) \leq \psi(u(s))$ for every $u\in\UU$ and $0\leq s\leq t < +\infty$, \label{itm: L2}
\item whenever the map $t\mapsto \psi(w(t)) \ (t\in\mathbb{R})$ is constant for some complete orbit $w$, then $w$ is stationary. \label{itm: L3}
\end{enumerate}
\end{definition}

\medskip
Generalized semiflows with Lyapunov function and continuous solutions are discussed in \cite{ball2004global, ball2000continuity}.
 
\medskip\noindent
\paragraph{Minimal solutions to generalized semiflows}

We find that any generalized semiflow with Lyapunov function and continuous solutions is a generalized $\Lambda$-semiflow, i.e. satisfies the hypotheses \ref{itm: 1} - \ref{itm: 5} in Definition \ref{def: sigma semiflow}. Moreover, the compactness hypothesis \ref{itm: C} is satisfied. 

We will see that the same holds good for any generalized semiflow with continuous solutions satisfying (\ref{eq: 3}). 

Also we will see that the presence of a function decreasing along solution curves allows of a further characterization of minimal solutions. \\ 

\begin{theorem}\label{thm: ms gs}
Let $\UU$ be a generalized semiflow on $\SS$. Suppose that there exists a function $\Psi: \SS \to \mathbb{R}$ for $\UU$ satisfying \ref{itm: L2} and \ref{itm: L3} and that every solution $u\in\UU$ is sequentially continuous, i.e.
\begin{displaymath}
u(t_j) \stackrel{\SS}{\to} u(t) \quad \text{whenever } t_j \to t, \ t_j, t \in [0, +\infty).
\end{displaymath} 
Then $\UU$ is a generalized $\Lambda$-semiflow, according to Definition \ref{def: sigma semiflow}, and satisfies the compactness hypothesis \ref{itm: C}. In particular, all the statements \ref{itm: T1} - \ref{itm: T5} of Theorem \ref{thm: minimal solution} hold good for $\UU$.  
\end{theorem} 

\paragraph{Comment on the function $\Psi: \SS\to \mathbb{R}$}

We suppose that there exists a function $\Psi: \SS\to \mathbb{R}$ for $\UU$ satisfying \ref{itm: L2} and \ref{itm: L3}. If, in addition, $\Psi$ is continuous, then it is called a Lyapunov function for $\UU$ (according to \cite{ball2000continuity, ball2004global}, Definition \ref{def: Lyapunov} above). 

Please note that we do not need to require continuity of $\Psi$ in order to obtain the results of Theorem \ref{thm: ms gs}. 
  
\begin{proof} 
The existence hypothesis \ref{itm: G1} implies that $\UU$ is nonempty. 

The hypotheses \ref{itm: 1} and \ref{itm: 2} correspond to \ref{itm: G2} and \ref{itm: G3}. In order to prove \ref{itm: 3}, it is now sufficient to show (\ref{eq: 3}), due to Remark \ref{rem: 3}. Let $u\in\UU$ and $0\leq s < t < +\infty$ such that $u(s) = u(t)$. Then it follows that $\Psi(u(r)) = \Psi(u(s))$ for all $r\in [s, t]$ since $\Psi\circ u$ is decreasing. Applying \ref{itm: G2}, \ref{itm: G3} and \ref{itm: G4}, we obtain that the map $v: \mathbb{R} \to \SS$ defined as 
\begin{displaymath}
v(r):= u(r + s - j(t-s)) \quad \text{if } r\in [j(t-s), (j+1)(t-s)], \ j\in\mathbb{Z},
\end{displaymath}
is a complete orbit for $\UU$. It holds that $\Psi(v(r)) = \Psi(u(s))$ for all $r\in\mathbb{R}$ and we may conclude that $v$ is stationary, i.e. $u(r) = u(s)$ for all $r\in [s, t]$. This proves (\ref{eq: 3}). 

Now, let us show that $\UU$ satisfies \ref{itm: 4}. Let $u\in\UU$. Suppose that there exists a map $w: [0, \theta) \to \SS$ with $\theta < + \infty$ and $w([0, \theta)) = \RR[u]$ such that $w|_{[0, T]}$ can be extended to a map in $\UU$ for every $T\in [0, \theta)$. In particular, whenever $T\in [0, \theta), \ S\in[0, +\infty), \ w(T) = u(S)$, the map $w(\cdot, T, S) : [0, +\infty) \to \SS$ defined as  
\begin{displaymath}
w(t, T, S) := \begin{cases}
w(t) &\text{ if } 0\leq t \leq T \\ 
u(t + S - T) &\text{ if } t > T
\end{cases}
\end{displaymath}
belongs to $\UU$. If $T_\star(u) < +\infty$, the claim easily follows from hypothesis \ref{itm: 3} already proved above. If $T_\star(u) = +\infty$, we select an increasing sequence of times $S_n \uparrow +\infty$. Due to \ref{itm: 3}, we find a corresponding increasing sequence $(T_n)$ with $w(T_n) = u(S_n)$; moreover $T_n\uparrow \theta$: indeed, if $\sup_n T_n \leq T < \theta$ for some $T\in (0, \theta)$, then $w$ would be constant in a small interval around $\sup_n T_n$ since
\begin{displaymath}
\bigcup_n w([0, T_n]) \ = \ \bigcup_n u([0, S_n]) \ = \ \RR[u] \ = \ w([0, \theta)), 
\end{displaymath} 
in contradiction to $T_\star(u) = +\infty$. 

Applying \ref{itm: G4} to $w_n(\cdot) := w(\cdot, T_n, S_n)$, we obtain that there exists a subsequence $n_k \uparrow +\infty$ and $\bar{w}\in\UU$ such that $w_{n_k}(t) \stackrel{\SS}{\to} \bar{w}(t)$ for all $t\geq 0$. It holds that $\bar{w}(t) = w(t)$ for all $t\in [0, \theta)$. As a member of $\UU$, the map $\bar{w}$ is sequentially continuous in $(0, +\infty)$. Hence the limit $\lim_{t\uparrow \theta} w(t)$ exists and coincides with $\bar{w}(\theta) =: w_\star\in\SS$. In particular, 
\begin{displaymath}
u(S_n) \ = \ w(T_n) \ \stackrel{\SS}{\to} \ w_\star \quad (n\to+\infty). 
\end{displaymath}
Since the sequence $S_n\uparrow +\infty$ has been chosen arbitrarily, it follows that 
\begin{displaymath}
u(t_n) \stackrel{\SS}{\to} w_\star \quad \text{whenever } t_n\to +\infty, \quad \bar{w}(t) = w_\star \quad \text{for all } t\geq \theta,  
\end{displaymath}
which gives \ref{itm: 4}. 

The hypothesis \ref{itm: 5} directly follows from a simple application of \ref{itm: G4}. 

\medskip\noindent
Finally, we prove \ref{itm: C}. Let a sequence $v_n\in\TT[\UU], \ n\in\mathbb{N},$ be given, satisfying $\sup_n \rho(v_n) < +\infty$ and $\RR[v_n] = \RR[v_1]$ for all $n\in\mathbb{N}$. We may assume w.l.o.g. that $T_n := \rho(v_n) \to T$ for some $T\in[0, +\infty)$. We select $\bar{v}_n\in\UU$ such that $\bar{v}_n(t) = v_n(t)$ for all $t\in[0, T_n]$. We note that $v_n(0) = v_1(0)$ and $v_n(T_n) = v_1(T_1)$ by \ref{itm: 3}. Due to \ref{itm: G4}, there exists a subsequence $n_k\uparrow +\infty$ and a solution $\bar{v} \in \UU$ such that $\bar{v}_{n_k}(t) \stackrel{\SS}{\to} \bar{v}(t)$ for all $t\in[0, +\infty)$. Since all the solutions are continuous in $(0, +\infty)$, this convergence is uniform in compact subsets of $(0, +\infty)$ by [\cite{ball2000continuity}, Thm. 2.2]. Moreover, it holds that
\begin{equation}\label{eq: thm 2.3}
\text{whenever} \quad \bar{v}_{n_k}(s_k)\in\RR[v_1], \ s_k\to0, \quad \text{then} \quad \bar{v}_{n_k}(s_k)\to\bar{v}(0).
\end{equation}
We prove (\ref{eq: thm 2.3}) (cf. proof of [\cite{ball2000continuity}, Thm. 2.3]): 

Suppose that $(\bar{v}_{n_k}(s_k))_k$ does not converge to $\bar{v}(0)$. Since $\RR[v_1]$ is sequentially compact, we may extract a convergent subsequence (still denoted by $\bar{v}_{n_k}(s_k)$) converging to some $\bar{v}_0\in\SS, \ \bar{v}_0\neq \bar{v}(0)$. For every $t > 0$, we have $\bar{v}_{n_k}(t + s_k) \stackrel{\SS}{\to} \bar{v}(t)$ by the uniform convergence in compact subsets of $(0, +\infty)$. Due to \ref{itm: G2} and \ref{itm: G4}, the map $w: [0, +\infty) \to \SS$, 
\begin{displaymath}
w(r):= \begin{cases}
\bar{v}_0 &\text{ if } r=0 \\
\bar{v}(r) &\text{ if } r > 0
\end{cases}
\end{displaymath} 
belongs to $\UU$. As $\bar{v}, w\in\UU$ are sequentially continuous in $[0, +\infty)$, we obtain $w(0) = \bar{v}(0)$, in contradiction to $\bar{v}_0 \neq \bar{v}(0)$. This proves (\ref{eq: thm 2.3}).

\medskip
It follows that $\bar{v}(T) = v_1(T_1)$ and $v_{n_k}(t) \stackrel{\SS}{\to} v(t)$ for all $t\in[0, +\infty)$, with $v\in\TT[\UU]$ defined by $v(t) := \bar{v}(t\wedge T)$ for all $t\geq 0$. Moreover, as $v_1$ is continuous, we have $\RR[v]\subset \RR[v_1]$, and by the uniform convergence, we obtain that $\RR[v_1]\subset \RR[v]$. Hence, $\RR[v] = \RR[v_1]$, and the proof is complete.   
\end{proof}

\begin{remark}
Following the proof of Theorem \ref{thm: ms gs} without assuming continuity of the solutions, it is not difficult to see that any generalized semiflow admitting a function $\Psi$ as above (i.e. for which \ref{itm: L2} and \ref{itm: L3} hold) satisfies the hypotheses \ref{itm: 1} - \ref{itm: 3}, \ref{itm: 5} and 
\begin{enumerate}[label = {({h4})}]
\item If $u\in\UU$ and there exists a map $w: [0, \theta) \to \SS$ with $\theta < +\infty$ such that $w|_{[0, T]}$ can be extended to a map in $\UU$ for every $T\in[0,\theta)$, and $w([0, \theta)) = \RR[u]$, 
then the $\omega$-limit set 
\begin{displaymath}
\omega(u):=\{w_\star\in\SS \ | \ \exists t_n \to +\infty, \ u(t_n)\stackrel{\SS}{\to} w_\star \}
\end{displaymath}
of $u$ is nonempty and there exists a map $\bar{w}: [0, +\infty) \to \SS$ in $\UU$ satisfying
\begin{displaymath}
\bar{w}(t) = w(t) \quad \text{if } t < \theta, \quad \bar{w}(t) \in \omega(u) \quad \text{if } t\geq \theta. 
\end{displaymath} 
\end{enumerate}
We notice that if $\Psi$ is continuous, then $\Psi$ is constant on $\omega(u)$.
\end{remark} 

We note that the only point in the proof of Theorem \ref{thm: ms gs} where the function $\psi$ plays a role is when we prove \ref{itm: 3}. Furthermore, the arguments in the proof of \ref{itm: 3} show that a generalized semiflow fails to satisfy \ref{itm: 3} if and only if it admits a nonconstant periodic orbit. So we obtain

\newpage
\begin{theorem}\label{thm: 3 neu}
Let $\UU$ be a generalized semiflow on $\SS$. 

If every solution $u\in\UU$ is sequentially continuous and satisfies (\ref{eq: 3}), then $\UU$ is a generalized $\Lambda$-semiflow satisfying the compactness hypothesis \ref{itm: C} and all the statements \ref{itm: T1} - \ref{itm: T5} of Theorem \ref{thm: minimal solution} hold good for $\UU$. 

If there exists a solution $u\in\UU$ which does not satisfy (\ref{eq: 3}), then there exists a nonconstant solution $v\in\UU$ and $\mu > 0$ such that $v(r) = v(r + \mu)$ for all $r\geq 0$. 
\end{theorem} 

Our next remark concerns the topological setting. 
\begin{remark}
The theory of generalized semiflows has been developed by Ball \cite{ball2000continuity, ball2004global} for metric spaces. 
The only (but critical) point where we make explicit use of the metrizability of the topology is when we apply [\cite{ball2000continuity}, Thm. 2.2] in the proof of the compactness hypothesis \ref{itm: C}.  
\end{remark}

\medskip
We conclude this section with a characterization of minimal solutions in terms of a function which decreases along solution curves. 

\begin{proposition}\label{prop: 1}
Let a topological space $\SS$ endowed with a Hausdorff topology be given. Let $\UU$ be a generalized $\Lambda$-semiflow on $\SS$ satisfying the compactness hypothesis \ref{itm: C}. Suppose that every solution $u\in\UU$ is sequentially continuous (\ref{eq: cont h}) and that there exists a function $\Psi: \SS \to \mathbb{R}$ which decreases along solution curves, i.e.
\begin{displaymath}
\Psi(u(t)) \ \leq \ \Psi(u(s)) \quad \text{for every} \quad 0\leq s < t <+\infty, \quad u\in\UU.  
\end{displaymath}
Then the following holds: 

\medskip
Whenever $u\in\UU_{\text{min}}, \ v\in\UU$ with $u\in\UU[\RR[v]]$, then $\Psi(u(t)) \leq \Psi(v(t))$ for all $t\in[0, +\infty)$. 

Whenever $u\in\UU$ and $\Psi$ is injective on $\RR[u]$, then $u$ belongs to $\UU_{\text{min}}$ if $\Psi(u(t)) \leq \Psi(v(t))$ for every $v\in\UU[\RR[u]]$ and $t\in[0, +\infty)$. 
\end{proposition}

\begin{proof}
The first statement directly follows from \ref{itm: T3} in Theorem \ref{thm: minimal solution}: indeed, if $t\in[0, T_\star(u))$, then there exists $\bar{t}\geq0$ with $u(t) = v(\bar{t})$ and applying \ref{itm: T3} we obtain $\bar{t} \geq t$ and hence $\Psi(u(t)) = \Psi(v(\bar{t})) \leq \Psi(v(t))$; if $T_\star(u) < +\infty$ and $t\geq T_\star(u)$, then $\Psi(u(t)) = \min_{s\geq 0} \Psi(u(s)) \leq \inf_{s\geq 0} \Psi(v(s)) \leq \Psi(v(t))$.  

Now, we prove the second statement. Let $u\in\UU$ be given and assume that $\Psi$ is injective on $\RR[u]$ and that $\Psi(u(t)) \leq \Psi(v(t))$ for every $v\in\UU[\RR[u]]$ and $t\in[0, +\infty)$. We note that $\Psi$ is injective on $\RR[v]$ for every $v\in\UU[\RR[u]]$: just suppose that there exist $v\in\UU[\RR[u]]$ with $T_\star(v) < +\infty$ and $\bar{t} < T_\star(v)$ with $\Psi(v(\bar{t})) = \Psi(v(T_\star(v))) = \min_{s\geq 0} \Psi(v(s))$, then $\Psi\circ v$ is constant in the interval $[\bar{t}, T_\star(v)]$, in contradiction to $\bar{t} < T_\star(v)$ and $\Psi$ injective on $\RR[u]$. The claim follows.
Suppose now that $u\succ v$ for some $v\in\UU$. Then there exists an increasing $1$-Lipschitz map $\sfz: [0, +\infty) \to [0, +\infty)$ such that $u(t) = v(\sfz(t))$ for all $t\in[0, +\infty)$. It holds that $\sfz(t) \leq t$ for all $t\geq 0$ and $v\in\UU[\RR[u]]$. Hence, $\Psi(u(t)) = \Psi(v(\sfz(t))) \geq \Psi(v(t)) \geq \Psi(u(t))$ for all $t\in[0, +\infty)$. This yields $u(t) = v(t)$ for all $t\in[0, +\infty)$ and the proof is complete.    
\end{proof}

\begin{remark}
We do not make use of \ref{itm: C} and (\ref{eq: cont h}) in the proof of the second statement of Proposition \ref{prop: 1} (cf. Remark \ref{rem: C}). 
\end{remark}

\section{Minimal solutions to gradient flows}\label{sec: gf} 

It is known that gradient flows can be studied within the framework of generalized semiflows \cite{rossi2011global}. However, our approach to apply the theory of minimal solutions to gradient flows in metric spaces is independent of Section \ref{sec: generalized semiflows}. The special structure of the energy dissipation inequality allows of taking into consideration cases in which the gradient flow for a functional does not fit into the concept of generalized semiflow due to the lack of compactness but still is a generalized $\Lambda$-semiflow.  

We find a particular feature of the minimal solutions to a gradient flow: they cross the critical set
\begin{displaymath}
\{x\in\SS \ | \ g(x) = 0\}
\end{displaymath}
of the functional with respect to the corresponding upper gradient $g$ only in a negligible set of times before they possibly become eventually constant.
 
\subsection{Curves of maximal slope}\label{subsec: gf1}
We give some of the basic definitions concerning gradient flows in metric spaces, following 
the fundamental book by Ambrosio, Gigli and Savar\'e \cite{AmbrosioGigliSavare05}: 
 
Let $(\SS, d)$ be a complete metric space and let the notation $\stackrel{\SS}{\to}$ correspond to the convergence in the metric $d$, i.e.
\begin{displaymath}
x_j \stackrel{\SS}{\to} x \quad :\Leftrightarrow \quad d(x_j, x) \to 0 
\end{displaymath}
for $x_j, x\in \SS$. 

So-called curves of maximal slope are defined for an extended real functional $\phi: \SS \rightarrow (-\infty, +\infty]$ with proper effective domain 
\begin{equation*}
D(\phi) := \{\phi < +\infty\} \ \neq \emptyset.
\end{equation*}
The notion of curves of maximal slope goes back to \cite{DeGiorgiMarinoTosques80}, with further developments in \cite{Degiovanni-Marino-Tosques85}, \cite{MarinoSacconTosques89}. 

\paragraph{Locally absolutely continuous curve}
\begin{definition}\label{def: lac curve}
We say that a curve $v: \ [0, +\infty) \rightarrow \SS$ is locally absolutely continuous and write $v\in AC_{\text{loc}}([0, +\infty);\SS)$ if there exists $m\in L_{\text{loc}}^1(0,+\infty)$ such that 
\begin{equation*}
d(v(s),v(t)) \leq \int^{t}_{s}{m(r) \ dr} \quad \text{for all } 0 \leq s\leq t < +\infty. 
\end{equation*}\\
In this case, the limit
\begin{equation*}
|v'|(t) := \mathop{\lim}_{s\to t} \frac{d(v(s),v(t))}{|s-t|}
\end{equation*}
exists for $\LL^1$-a.e. $t\in (0, +\infty)$, the function $t \mapsto |v'|(t)$ belongs to $L_{\text{loc}}^1(0, +\infty)$ and is called the metric derivative of $v$. The metric derivative is $\LL^1$-a.e. the smallest admissible function $m$ in the definition above. 
\end{definition}

\paragraph{Strong upper gradient}
 
\begin{definition}\label{def: sug}
A function $g: \SS\to [0, +\infty]$ is a strong upper gradient for the functional $\phi$ if for every $v\in AC_{\text{loc}}([0,+\infty);\SS)$ the function $g\circ v$ is Borel and 
\begin{equation}\label{eq: sug}
|\phi(v(t)) - \phi(v(s))| \leq \int^{t}_{s}{g(v(r))|v'|(r) \ dr} \quad \text{for all } 0 \leq s\leq t < +\infty.
\end{equation}
In particular, if $g\circ v |v'| \in L_{\text{loc}}^1(0,+\infty)$ then $\phi \circ v$ is locally absolutely continuous and 
\begin{equation*}
|(\phi \circ v)'(t)| \leq g(v(t))|v'|(t) \quad \text{for } \LL^1\text{-a.e. } t\in (0,+\infty).
\end{equation*}
\end{definition}
This slightly modified version of [\cite{AmbrosioGigliSavare05}, Def. 1.2.1] (which requires (\ref{eq: sug}) only for $s > 0$) can be found in \cite{rossi2011global}. 
 
In \cite{AmbrosioGigliSavare05}, also the concept of \textit{weak upper gradient} is defined. The notion of upper gradient is an abstraction of the modulus of the gradient to a general metric and nonsmooth setting. 

\paragraph{Curve of maximal slope}
\begin{definition}\label{def: curve of maximal slope}
Let $g: \SS \to [0, +\infty]$ be a strong or weak upper gradient for the functional $\phi$. 

A locally absolutely continuous curve $u: [0, +\infty) \to \SS$ is called a curve of maximal slope for $\phi$ with respect to its upper gradient $g$ if $\phi\circ u$ is $\LL^1$-a.e. equal to a decreasing map $\varphi: [0, +\infty) \to \mathbb{R}$, i.e.
\begin{displaymath}
\phi(u(r)) = \varphi(r) \text{ for } \LL^1\text{-a.e. } r\geq 0, \quad \varphi(t) \leq \varphi(s) \text{ for all } 0\leq s < t <+\infty,
\end{displaymath}
and the energy dissipation inequality  
\begin{equation*}
\varphi(s) - \varphi(t) \ \geq \ \frac{1}{2} \int^{t}_{s}{g^2 (u(r)) \ dr} \ + \ \frac{1}{2} \int^{t}_{s}{|u'|^2 (r) \ dr}  
\end{equation*}  
is satisfied for all $0\leq s \leq t < +\infty$.
\end{definition}


\medskip
Typical candidates for $g$ are the \textit{local slope}
\begin{equation*}
|\partial\phi|(x) := \mathop{\limsup}_{d(y,x) \to 0} \frac{(\phi(x)-\phi(y))^+}{d(x,y)} \quad (x\in D(\phi)), 
\end{equation*} 
the \textit{relaxed slope}
\begin{equation*}
|\partial^- \phi|(x) := \inf  \left\{\mathop{\liminf}_{j\to \infty} |\partial\phi|(x_j): \ d(x_j,x)\to 0, \ \sup_{j} \phi(x_j) < +\infty \right\} 
\end{equation*}
and similar modifications of the lower semicontinuous envelope of the local slope \cite{AmbrosioGigliSavare05, RossiSavare06, Degiovanni-Marino-Tosques85, MarinoSacconTosques89, rossi2011global}. 
\begin{remark}
If $\SS = \mathbb{R}^d$ and $\phi: \mathbb{R}^d \to \mathbb{R}$ is a continuously differentiable Lipschitz function, then $g:=|\nabla\phi| = |\partial\phi| = |\partial^-\phi|$ is a strong upper gradient for $\phi$, the energy dissipation inequality is equivalent to the classical gradient flow equation
\begin{displaymath}
u'(t) \ = \ -\nabla\phi(u(t)), \quad t > 0, 
\end{displaymath}
and admits at least one solution for every initial value. 
\end{remark} 

\paragraph{Definition of $\UU(\phi, g)$}

It is usually not clear a priori whether a candidate function $g: \SS \to [0, +\infty]$ is an upper gradient or not (except that the local slope is a weak upper gradient \cite{AmbrosioGigliSavare05}). 

Our analysis of gradient flows with regard to our concept of generalized $\Lambda$-semiflow and minimal solutions will not rely on the behaviour of $g$ as a strong or weak upper gradient. Our considerations will concern locally absolutely continuous curves satisfying the energy dissipation inequality for some given function $g: \SS \to [0, +\infty]$ without specifying the role $g$ plays for the functional $\phi$.

In view of the concatenation hypothesis \ref{itm: 2}, we assume that the energy dissipation inequality holds everywhere for $\varphi = \phi\circ u$.

\begin{definition}\label{def: U(phi, g)}
Let $\phi: \SS \to (-\infty, +\infty]$ and $g: \SS \to [0, +\infty]$ be given. We define $\UU(\phi, g)$ as the family of all the locally absolutely continuous curves $u\in AC_{\text{loc}}([0, +\infty); \SS)$ with $u(0)\in D(\phi)$, satisfying the energy dissipation inequality 
\begin{equation}\label{eq: EDI'}
\phi(u(s)) - \phi(u(t)) \ \geq \ \frac{1}{2} \int^{t}_{s}{g^2 (u(r)) \ dr} \ + \ \frac{1}{2} \int^{t}_{s}{|u'|^2 (r) \ dr} 
\end{equation}
for all $0\leq s \leq t < +\infty$. 
\end{definition}

If $g$ is a weak or strong upper gradient for $\phi$ and $u\in\UU(\phi, g)$, then $u$ is a curve of maximal slope for $\phi$ w.r.t. $g$. 

\begin{remark}\label{rem: tacit}
In Definition \ref{def: U(phi, g)}, we tacitly assume that $g\circ u$ is Borel; otherwise the integral on the right-hand side would be set $+\infty$.
\end{remark}

\paragraph{Example of a nonempty family $\UU(\phi, g)$}

The following existence result is provided in \cite{AmbrosioGigliSavare05}, the proof of which is based on the notion of minimizing movements \cite{DeGiorgi93}: Suppose that the functional $\phi: \SS\to (-\infty, +\infty]$ is lower semicontinuous, i.e.
\begin{equation}\label{eq: AGS lsc}
d(x_j, x) \to 0 \quad  \Rightarrow \quad \mathop{\liminf}_{j\to \infty} \phi(x_j) \geq \phi(x),
\end{equation}
quadratically bounded from below, i.e. there exist $A, B > 0, \ x_{\star}\in \SS$ such that 
\begin{equation}\label{eq: AGS coercive}  
\phi(\cdot) \geq -A -B d^2(\cdot, x_{\star}),
\end{equation}
and suppose that $d$-bounded subsets of a sublevel of $\phi$ are relatively compact, i.e. 
\begin{equation}\label{eq: AGS compact}
\sup_{j,l}\{d(x_j,x_l), \phi(x_j)\} < +\infty \ \Rightarrow \ \exists \ j_k \uparrow +\infty, \ x\in\SS: \ d(x_{j_k}, x) \to 0. 
\end{equation}
Further, suppose that $g:= |\partial^-\phi|$ is a strong upper gradient for $\phi$. Then the following holds \cite{AmbrosioGigliSavare05}: for every $u_0\in D(\phi)$, there exists at least one curve $u$ of maximal slope for $\phi$ w.r.t. $|\partial^-\phi|$, with initial value $u(0) = u_0$, the energy dissipation inequality (\ref{eq: EDI'}) holds (in fact, equality holds in (\ref{eq: EDI'})) and $u\in\UU(\phi, |\partial^-\phi|)$. 


\begin{remark}\label{rem: sug}
Whenever $g: \SS\to[0, +\infty]$ is a strong upper gradient for a functional $\phi: \SS \to (-\infty, +\infty]$, and there exists a curve $u$ of maximal slope for $\phi$ w.r.t. $g$, it follows from Definition \ref{def: sug} of strong upper gradient that $u\in\UU(\phi, g)$ (with equality in (\ref{eq: EDI'})) and $\phi\circ u$ is locally absolutely continuous. 

The family $\UU(\phi, g)$ then coincides with the collection of all the curves of maximal slope for $\phi$ w.r.t. $g$.
\end{remark}

\subsection{Gradient flow as generalized $\Lambda$-semiflow} \label{sec: 42}

We want to prove that $\UU(\phi, g)$ is a generalized $\Lambda$-semiflow: 

\begin{theorem}\label{thm: gf as gls}
Let $\phi: \SS\to (-\infty, +\infty]$ and $g: \SS\to [0, +\infty]$ be given. 
We assume that $\phi$ and $g$ are lower semicontinuous, i.e.
\begin{equation}\label{eq: phi g lsc}
d(x_j, x)\to 0 \quad \Rightarrow \quad \liminf_{j\to+\infty} \phi(x_j) \geq \phi(x), \quad \liminf_{j\to+\infty} g(x_j)\geq g(x), 
\end{equation}
and $\phi$ is quadratically bounded from below, i.e. there exist $A, B > 0, \ x_\star\in\SS$ such that
\begin{equation}\label{eq: phi quadratic}
\phi(\cdot) \geq -A -Bd^2(\cdot, x_\star), 
\end{equation}
and we suppose that $\UU(\phi, g)\neq \emptyset$. Then $\UU(\phi, g)$ is a generalized $\Lambda$-semiflow, according to Definition \ref{def: sigma semiflow}. 
\end{theorem}

\begin{proof}
We first note that if $g$ is lower semicontinuous, then $g\circ u$ is Borel for every curve $u\in AC_{\text{loc}}([0, +\infty); \SS)$. 

The hypothesis \ref{itm: 1} follows by the classical change of variables formula: if $u\in\UU(\phi, g)$ and $\tau\geq 0$, then $u_\tau(\cdot) := u(\cdot + \tau) \in AC_{\text{loc}}([0, +\infty);\SS)$ with metric derivative $|u_\tau'|(\cdot) = |u'|(\cdot + \tau)$ and 
\begin{eqnarray*}
\phi(u_\tau(s)) - \phi(u_\tau(t)) &\geq& \frac{1}{2}\int_{s+\tau}^{t+\tau}{g^2(u(r)) \ dr} + \frac{1}{2}\int_{s+\tau}^{t+\tau}{|u'|^2(r) \ dr} \\
&\geq& \frac{1}{2}\int_{s}^{t}{g^2(u_\tau(r)) \ dr} + \frac{1}{2}\int_{s}^{t}{|u_\tau'|^2(r) \ dr}. 
\end{eqnarray*} 

Similarly, we show \ref{itm: 2}. Let $u, v\in\UU(\phi, g)$ with $v(0) = u(\bar{t})$ for some $\bar{t}\geq 0$ and define $w: [0, +\infty) \to \SS$, 
\begin{displaymath}
w(t) := \begin{cases}
u(t) &\text{ if } t\leq \bar{t} \\
v(t-\bar{t}) &\text{ if } t > \bar{t}
\end{cases}
\end{displaymath} 
Clearly, $w\in AC_{\text{loc}}([0, +\infty); \SS)$ with 
\begin{displaymath}
|w'|(r) = \begin{cases}
|u'|(r) &\text{ if } r\leq \bar{t} \\
|v'|(r-\bar{t}) &\text{ if } r > \bar{t}
\end{cases}
\end{displaymath}
and the energy dissipation inequality (\ref{eq: EDI'}) directly follows for $0\leq s \leq t \leq \bar{t}$ and by change of variable as above, for $\bar{t} \leq s \leq t < +\infty$. If $0\leq s < \bar{t} < t$, we obtain (\ref{eq: EDI'}) by splitting up
\begin{displaymath}
\phi(w(s)) - \phi(w(t)) \ = \ \phi(w(s)) - \phi(w(\bar{t})) + \phi(w(\bar{t})) - \phi(w(t)). 
\end{displaymath} 
This shows \ref{itm: 2}. 

Now, let a map $u: [0, +\infty) \to \SS$ be given with the property that $u|_{[0, T]}$ can be extended to a map in $\UU(\phi, g)$ for all $T > 0$, i.e. for every $T > 0$ there exists $w_T \in \UU(\phi, g)$ with 
$w_T(t) = 
u(t) \text{ if } t\leq T$. 
In particular, it holds that $u\in AC_{\text{loc}}([0, +\infty); \SS)$ and $|w_T'|(\cdot) = |u'|(\cdot)$ in $(0, T)$. Hence, $u\in\UU(\phi, g)$. This shows that $\UU(\phi, g)$ satisfies \ref{itm: 5}.    

Obviously, $\UU(\phi, g)$ satisfies \ref{itm: 3}. 

\medskip
It remains to prove \ref{itm: 4}. Let $u\in\UU(\phi, g)$. Suppose that there exists a map $w: [0, \theta) \to \SS$ with $\theta < +\infty$ and $w([0, \theta)) = \RR[u]$ such that $w|_{[0, T]}$ can be extended to a map in $\UU(\phi, g)$ for all $T\in[0, \theta)$. Then $w\in AC([0, T]; \SS)$ for every $T \in (0, \theta)$, i.e. the metric derivative 
\begin{displaymath}
|w'|(t):= \mathop{\lim}_{s\to t} \frac{d(w(s),w(t))}{|s-t|}
\end{displaymath} 
exists for $\LL^1$-a.e. $t\in(0, \theta)$ and $t\mapsto |w'|(t)$ belongs to $L^1([0, T]; \SS)$ for every $T\in(0, \theta)$, and 
\begin{displaymath}
d(w(s),w(t)) \leq \int^{t}_{s}{|w'|(r) \ dr} \quad \text{for all } 0 \leq s\leq t < \theta. 
\end{displaymath}
Moreover, $w(0) \in D(\phi)$ and the energy dissipation inequality 
\begin{equation}\label{eq: EDI w}
\phi(w(s)) - \phi(w(t)) \ \geq \ \frac{1}{2} \int_s^t{g^2(w(r)) \ dr} \ + \ \frac{1}{2}\int_s^t{|w'|^2(r) \ dr}
\end{equation}
holds for all $0\leq s \leq t < \theta$. By assumption, there exist $A, B > 0, \ x_\star\in\SS$ such that
\begin{displaymath}
\phi(\cdot) \geq -A - B d^2(\cdot, x_\star).
\end{displaymath}
We set 
\begin{displaymath}
\xi(t):= \phi(w(t)) + 2B d^2(w(t), x_\star) + A \quad \text{for } t\in[0, \theta).
\end{displaymath} 
It holds that $\xi$ is nonnegative and Borel (since $\phi$ is lower semicontinuous) and for every $t\in[0, \theta)$, the map $\xi$ is bounded from above in $[0, t]$ and  
\begin{eqnarray*}
\xi(t) &\leq& \xi(0) - \frac{1}{2} \int_0^t{|w'|^2(r) \ dr}  + 4B \int_0^t{d(w(r), x_\star)|w'|(r) \ dr} \\
&\leq& \xi(0) + 8B^2 \int_0^t{d^2(w(r), x_\star) \ dr} \\
&\leq& \xi(0) + 8B\int_0^t{\xi(r) \ dr}. 
\end{eqnarray*}
We used the fact that $[0, t] \ni r \mapsto d^2(w(r), x_\star)$ is absolutely continuous due to the chain rule for BV functions [\cite{Ambrosio-Fusco-Pallara00}, Thm. 3.99]: indeed, the map $[0, t] \ni r\mapsto \eta(r):=d(w(r), x_\star)$ is absolutely continuous with
\begin{displaymath}
|\eta(r_2) - \eta(r_1)| \leq d(w(r_1), w(r_2)) \leq \int_{r_1}^{r_2}{|w'|(r) \ dr} \quad \text{ for } 0\leq r_1 \leq r_2 \leq t 
\end{displaymath}
and bounded in $[0, t]$, so we may apply [\cite{Ambrosio-Fusco-Pallara00}, Thm. 3.99] to $\eta^2$ and obtain
\begin{displaymath}
|d^2(w(r_2), x_\star) - d^2(w(r_1), x_\star)| \leq \int_{r_1}^{r_2}{2 d(w(r), x_\star)|w'|(r) \ dr} \ (0\leq r_1 \leq r_2 \leq t). 
\end{displaymath}

\medskip
Applying the integral form of Gronwall's inequality (see e.g. [\cite{Evans2010}, Appendix B]) to $\xi$, we obtain
\begin{equation}\label{eq: xi}
\xi(t) \ \leq \ \xi(0)(1+ 8Bt e^{8Bt}) \ \leq \ \underbrace{\xi(0)(1+8B\theta e^{8B\theta})}_{=: \xi_{0, \theta}} \quad \text{ for all } t\in[0, \theta). 
\end{equation}

In particular, 
\begin{equation}\label{eq: w bounded}
Bd^2(w(t), x_\star) \ \leq \ \xi(t) \ \leq \ \xi_{0, \theta}, \quad \phi(w(t)) \ \geq \ -A - B\xi_{0, \theta} \ > \ -\infty
\end{equation}
for all $t\in[0, \theta)$. By H\"older inequality, it follows from (\ref{eq: EDI w}) and (\ref{eq: w bounded}) that 
\begin{displaymath}
d(w(s), w(t)) \ \leq \ (t-s)^{\frac{1}{2}} (2\phi(w(0)) + 2A + 2B\xi_{0, \theta})^{\frac{1}{2}}  \quad \text{for all } 0\leq s \leq t < \theta. 
\end{displaymath}
Since $\SS$ is complete, this shows that the limit $\lim_{t\uparrow \theta} w(t) =: w_\star\in\SS$ exists.

If $T_\star(u) < +\infty$, then $u(t) = w_\star$ for all $t\in[T_\star(u), +\infty)$ and there exists $T\in[0, \theta)$ such that $w(t) = w_\star$ for all $T\leq t < \theta$; nothing remains to be shown in this case. 

If $T_\star(u) = +\infty$ and $S_n\uparrow +\infty$, there exists a corresponding increasing sequence $T_n\uparrow \theta$ with $w(T_n) = u(S_n)$; this follows from \ref{itm: 3} (cf. proof of Theorem \ref{thm: ms gs}). So we obtain
\begin{displaymath}
d(u(t_n), w_\star) \to 0 \quad \text{whenever } t_n \to +\infty. 
\end{displaymath}
Moreover, since $u\in\UU(\phi, g)$ satisfies the energy dissipation inequality (\ref{eq: EDI'}) and $\inf_{t\geq 0} u(t) = \inf_{t\in[0, \theta)}w(t) > -\infty$ by (\ref{eq: w bounded}), it holds that 
\begin{displaymath}
\int_0^{+\infty}{g^2(u(r)) \ dr} < +\infty. 
\end{displaymath}
Hence, $\liminf_{r\to+\infty} g(u(r)) = 0$ and we obtain
\begin{equation}\label{eq: g(w_star) = 0}
g(w_\star) = 0
\end{equation}
by the lower semicontinuity of $g$. Further, for $s\in[0, \theta)$, the energy dissipation inequality
\begin{equation}\label{eq: EDI w_star}
\phi(w(s)) - \phi(w_\star) \ \geq \ \frac{1}{2} \int_s^\theta{g^2(w(r)) \ dr} \ + \ \frac{1}{2}\int_s^\theta{|w'|^2(r) \ dr}
\end{equation}
follows from (\ref{eq: EDI w}) and the lower semicontinuity of $\phi$. 

We define $\bar{w}: [0, +\infty) \to \SS$, 
\begin{displaymath}
\bar{w}(t):= \begin{cases}
w(t) &\text{ if } t < \theta \\
w_\star &\text{ if } t\geq \theta
\end{cases}
\end{displaymath}
Clearly, $\bar{w}\in AC_{\text{loc}}([0, +\infty); \SS)$, and by (\ref{eq: EDI w}), (\ref{eq: EDI w_star}) and (\ref{eq: g(w_star) = 0}), it holds that $\bar{w}\in\UU(\phi, g)$. 

The proof is complete. 
\end{proof}

The assumptions (\ref{eq: phi g lsc}) and (\ref{eq: phi quadratic}) on $\phi$ and $g$ in Theorem \ref{thm: gf as gls} are only used in the proof of \ref{itm: 4}. The lower semicontinuity hypotheses on $\phi$ and $g$ allow the passage to the limit in the energy dissipation inequality and are natural assumptions whenever some kind of limit behaviour concerning the energy dissipation inequality is of interest (cf. the long-time analysis for gradient flows in \cite{rossi2011global}). This will again be the case in the proof of \ref{itm: C} in Section \ref{sec: 43}. 

We note that we do not need to require any compactness property of $\phi$ such as (\ref{eq: AGS compact}); the quadratic bound (\ref{eq: phi quadratic}) from below suffices for our purposes. 

Also the existence of minimal solutions will be proved without assuming any compactness property of $\phi$. 

\subsection{Minimal gradient flow} \label{sec: 43}

In this section, we study minimal solutions to gradient flows. Our aim is to apply Theorem \ref{thm: minimal solution} and Proposition \ref{prop: 1} to $\UU = \UU(\phi, g)$, providing existence and characteristics of minimal solutions. Moreover, we will see that minimal solutions $u\in\UU_{\text{min}}$ to gradient flows are characterized by the particular property that
\begin{displaymath}
\LL^1(\{t\in[0, T_\star(u)) \ | \ g(u(t)) = 0\}) = 0. 
\end{displaymath}  

\medskip
Let $\phi: \SS\to (-\infty, +\infty]$ and $g:\SS\to [0, +\infty]$ be given. Throughout this section, we assume that $\phi$ and $g$ are lower semicontinuous (\ref{eq: phi g lsc}) 
and $\phi$ is quadratically bounded from below (\ref{eq: phi quadratic}), 
and we define $\UU(\phi, g)$ as in Definition \ref{def: U(phi, g)}. Due to Theorem \ref{thm: gf as gls}, the family $\UU(\phi, g)$ is a generalized $\Lambda$-semiflow on $\SS$ (provided it is nonempty). 

\medskip
We want to prove that $\UU(\phi, g)$ satisfies \ref{itm: C}. The critical point is a passage to the limit in the energy dissipation inequality, as in the proof of \ref{itm: 4}. The passage to the limit will now concern both terms on the left-hand side of the energy dissipation inequality (\ref{eq: EDI'}) so that the lower semicontinuity of $\phi$ will not suffice. Such obstacles are usually overcome by assuming that $g$ is a strong upper gradient for $\phi$ or by allowing any decreasing function $\varphi \geq \phi\circ u$ in a modified energy dissipation inequality with pairs $(u, \varphi)$ as solutions (cf. \cite{rossi2011global}). 

For our purposes, it is sufficient to assume that $\phi\circ u: [0, +\infty) \to \mathbb{R}$ is continuous for every solution $u\in\UU(\phi, g)$. This is satisfied e.g. if $g$ is a strong upper gradient (cf. Remark \ref{rem: sug}).  

\begin{theorem}\label{thm: 59}
Let the assumptions of Theorem \ref{thm: gf as gls} be satisfied and suppose that 
\begin{equation}\label{eq: phicircu}
\phi\circ u: [0, +\infty) \to \mathbb{R} \quad \text{is continuous for every} \quad u\in\UU(\phi, g). 
\end{equation}

Then the generalized $\Lambda$-semiflow $\UU(\phi, g)$ satisfies \ref{itm: C} and all the statements \ref{itm: T1} - \ref{itm: T5} of Theorem \ref{thm: minimal solution} hold good for $\UU(\phi, g)$. Moreover, both statements of Proposition \ref{prop: 1} are applicable to $\UU(\phi, g)$.  
\end{theorem}
\begin{proof}
We write $\UU=\UU(\phi, g)$. Let a sequence $v_n\in\TT[\UU], \ n\in\mathbb{N},$ be given with $\sup_n\rho(v_n) < +\infty$ and $\RR[v_n] = \RR[v_1]$ for all $n\in\mathbb{N}$. Since the truncated solution $v_1$ is continuous with $T_1:=\rho(v_1) < +\infty$, its range $\RR[v_1]$ is sequentially compact. Furthermore, it is straightforward to check that
\begin{displaymath}
\sup_{n\in\mathbb{N}}\int_0^{+\infty}{|v_n'|^2(r) \ dr} \ \leq \ 2(\phi(v_1(0)) - \phi(v_1(T_1))) \ < \ +\infty.   
\end{displaymath} 
Applying a refined version of Ascoli-Arzel\`a theorem [\cite{AmbrosioGigliSavare05}, Prop. 3.3.1], we obtain that there exist a subsequence $n_k\uparrow +\infty$ and a curve $v: [0, +\infty) \to \SS$ such that 
\begin{equation}\label{eq: 111}
v_{n_k}(t) \stackrel{\SS}{\to} v(t) \quad \text{for all } t\in[0, +\infty). 
\end{equation}
It is not difficult to see that $v\in AC_{\text{loc}}([0, +\infty); \SS)$ and 
\begin{equation}\label{eq: 112}
\int_{s}^{t}{|v'|^2 (r) \ dr} \ \leq \ \liminf_{k\to +\infty}\int_s^t{|v_{n_k}'|^2(r) \ dr} \quad \text{for all } 0\leq s\leq t <+\infty. 
\end{equation}

We may assume w.l.o.g. that $T_{n_k} := \rho(v_{n_k}) \to T$ for some $T\in[0, +\infty)$. For every $t\in[0, +\infty)$, there exists a sequence of times $t_k\in[0, T_1]$ such that $v_1(t_k) = v_{n_k}(t)$. It follows from this and from (\ref{eq: 111}) and (\ref{eq: phicircu}) that $\RR[v]\subset \RR[v_1]$ and 
\begin{displaymath}
\phi(v_{n_k}(t)) \to \phi(v(t)) \quad \text{for all } t\in[0, +\infty), \quad v(t) = v_1(T_1) \quad \text{for all } t\geq T.   
\end{displaymath} 
We obtain 
\begin{eqnarray*}
\phi(v(s)) - \phi(v(t)) &=& \lim_{k\to +\infty} (\phi(v_{n_k}(s)) - \phi(v_{n_k}(t))) \\
&\geq& \liminf_{k\to +\infty} \frac{1}{2}\int_s^t{g^2(v_{n_k}(r)) \ dr} + \liminf_{k\to+\infty} \frac{1}{2} \int_s^t{|v_{n_k}'|^2(r) \ dr} \\
&\geq& \frac{1}{2}\int_s^t{g^2(v(r)) \ dr} + \frac{1}{2} \int_s^t{|v'|^2(r) \ dr} 
\end{eqnarray*}
for all $0\leq s\leq t < T$, due to the fact that $v_{n_k}$ satisfies (\ref{eq: EDI'}) in $[0, T_{n_k}]$ and due to (\ref{eq: 112}), the lower semicontinuity (\ref{eq: phi g lsc}) of $g$ and Fatou's lemma.

Since $v$ is continuous, $\RR[v]\subset\RR[v_1]$ and (\ref{eq: phicircu}) holds, the map $\phi\circ v$ is continuous. It follows that $\RR[v] = \RR[v_1]$ since $v(0) = v_1(0)$, $v(T) = v_1(T_1)$, $\phi(\RR[v_1]) = [\phi(v_1(T_1)), \phi(v_1(0))]$ and $\phi$ injective on $\RR[v_1]$ (cf. Remark \ref{rem: injective}); further, the energy dissipation inequality
\begin{displaymath}
\phi(v(s)) - \phi(v(t)) \ \geq \ \frac{1}{2}\int_s^t{g^2(v(r)) \ dr} + \frac{1}{2} \int_s^t{|v'|^2(r) \ dr}
\end{displaymath}   
holds for all $0\leq s \leq t \leq T$. 

Now, let $\bar{v}_1\in\UU$ such that $v_1(\cdot) = \bar{v}_1(\cdot\wedge T_1)$. Similar arguments as in the proof of Theorem \ref{thm: gf as gls}, \ref{itm: 2}, show that $\bar{v}\in\UU$, where $\bar{v}: [0, +\infty) \to \SS$ is defined as 
\begin{displaymath}
\bar{v}(t) := \begin{cases}
v(t) &\text{ if } 0\leq t \leq T \\
\bar{v}_1(t-T + T_1) &\text{ if } t > T
\end{cases}
\end{displaymath}
Hence $v\in\TT[\UU]$. The proof of \ref{itm: C} is complete. 
\end{proof}
\begin{remark}\label{rem: injective}
Proposition \ref{prop: 1} is applicable with $\psi:= \phi$; it is true that $\phi$ may take the value $+\infty$ but for the statements of Proposition \ref{prop: 1} to hold good, it suffices that $\phi(u(t))\leq \phi(u(s)) < +\infty$ for all $0\leq s\leq t < +\infty$, $u\in\UU(\phi, g)$. We note that $\phi$ is injective on $\RR[u]$ for every $u\in\UU(\phi, g)$. To be more precise: 
The energy dissipation inequality (\ref{eq: EDI'}) implies that for every $u\in\UU(\phi, g)$ and $0\leq s \leq t < +\infty$ the following four points are equivalent: 
\begin{enumerate}[label={ ({\roman*}) }]
\item $\phi(u(s)) = \phi(u(t))$, 
\item $|u'|(r) = 0$ for $\LL^1$-a.e. $r\in(s, t)$, 
\item $u(r) = u(s)$ for all $r\in[s, t]$, 
\item $u(s) = u(t)$. 
\end{enumerate}
\medskip
Moreover, we note that for $\UU=\UU(\phi, g)$ and the range $\RR=\RR[y]\subset \SS$ of a solution $y\in\UU$, it holds that 
\begin{displaymath}
w\in\UU[\RR] \quad \Leftrightarrow \quad w\in\UU, \ \RR\subset \RR[w] \subset \overline{\RR}. 
\end{displaymath}
\end{remark}

\medskip
If $g$ is a strong upper gradient for $\phi$, then $\UU(\phi, g)$ coincides with the collection of all the curves of maximal slope for $\phi$ w.r.t. $g$ (cf. Remark \ref{rem: sug}). 
The next proposition deals with a special quality of minimal solutions to a gradient flow in terms of the $0$ level set of the corresponding strong upper gradient. 
\begin{proposition}\label{prop: cross critical set} (cf. [\cite{sf2017}, Thm. 3.4 (5)])
Let the assumptions of Theorem \ref{thm: gf as gls} be satisfied and suppose that $g$ is a strong upper gradient for $\phi$. Then the following two statements are equivalent for a solution $u\in\UU(\phi, g)$: 
\begin{enumerate}[label=({\roman*})]
\item $u$ is minimal, \label{itm: prop i}
\item \label{itm: prop ii} $u$ crosses the set $\{x\in\SS \ | \ g(x) = 0\}$ of critical points of $\phi$ w.r.t. its upper gradient $g$ in an $\LL^1$-negligible set of times, i.e. 
\begin{equation}\label{eq: prop ii}
\LL^1(\{t\in[0, T_\star(u)): \ g(u(t)) = 0\}) = 0. 
\end{equation}
\end{enumerate} 
\end{proposition}

\begin{proof}
First we notice some properties of $\UU(\phi, g)$ if $g$ is a strong upper gradient (cf. Remark \ref{rem: sug}): every solution $u\in\UU(\phi, g)$ satisfies 
\begin{equation}\label{eq: EDE}
\phi(u(s)) - \phi(u(t)) \ = \ \frac{1}{2}\int_s^t{g^2(u(r)) \ dr} \ + \ \frac{1}{2}\int_s^t{|u'|^2(r) \ dr}
\end{equation}
for all $0\leq s\leq t < +\infty$, it holds that
\begin{equation}\label{eq: g u = u'}
g(u(r)) \ = \ |u'|(r) \quad \text{for } \LL^1\text{-a.e. } r\in[0, +\infty),  
\end{equation}
and $\phi\circ u$ is locally absolutely continuous with 
\begin{equation}\label{eq: (phicircu)'}
(\phi\circ u)'(r) \ = \ -g^2(u(r)) \ = \ -|u'|^2(r) \quad \text{for } \LL^1\text{-a.e. } r\in[0, +\infty). 
\end{equation}

\medskip
Let us show that \ref{itm: prop ii} implies \ref{itm: prop i}. Let $u\in\UU$ satisfy (\ref{eq: prop ii}) and suppose that $u\succ v$ for some $v\in\UU(\phi, g)$. Then $\RR[v]\subset\overline{\RR[u]}$ and there exists an increasing $1$-Lipschitz map $\sfz: [0, +\infty) \to [0, +\infty)$ with $\sfz(0) = 0$ such that $u(t) = v(\sfz(t))$ for all $t\geq 0$. The map $\sfz$ is differentiable $\LL^1$-a.e. in $[0, +\infty)$ and the chain rule for absolutely continuous functions (see e.g. [\cite{leoni2009first}, Thm. 3.44]) and (\ref{eq: (phicircu)'}) yield 
\begin{displaymath}
g^2(u(r)) = -(\phi\circ u)'(r)  =  -(\phi\circ v \circ \sfz)'(r)  =  g^2(v(\sfz(r)))\sfz'(r)  =  g^2(u(r)) \sfz'(r) 
\end{displaymath} 
for $\LL^1$-a.e. $r\in[0, +\infty)$. By (\ref{eq: prop ii}), it follows that $\sfz'(r) = 1$ for $\LL^1$-a.e. $r\in[0, T_\star(u))$, which implies $\sfz(t) = t$ for all $t\in[0, T_\star(u))$. This shows that $u = v$ and the claim is proved. 

\medskip
Now, we prove that \ref{itm: prop i} implies \ref{itm: prop ii}. Let $u$ be a minimal solution, with $T_\star(u)\in(0, +\infty]$. 
Let 
\begin{displaymath}
\Omega \ := \ \{t\in(0, T_\star(u)) \ : \ g(u(t)) > 0\}. 
\end{displaymath}
As $g$ is lower semicontinuous (\ref{eq: phi g lsc}), the set $\Omega$ is open. 

We define $\sfx: [0, T_\star(u)) \to [0, +\infty)$ as 
\begin{displaymath}
\sfx(t) := \int_0^t{|u'|(r) \ dr}. 
\end{displaymath} 
The map $\sfx$ is locally absolutely continuous; further, it is strictly increasing since $u$ is injective in $[0, T_\star(u))$ by Theorem \ref{thm: minimal solution}, \ref{itm: T2}. Let 
\begin{displaymath}
X:= \lim_{t\uparrow T_\star(u)}\sfx(t) = \int_0^{T_\star(u)}{|u'|(r) \ dr} \in(0, +\infty].
\end{displaymath}
There exists a strictly increasing, continuous inverse $\sfy: [0, X) \to [0, T_\star(u)),$
\begin{displaymath}
\sfy(\sfx(t))  =  t \quad \text{for all } t\in[0, T_\star(u)), \quad \sfx(\sfy(\rmx)) = \rmx \quad \text{for all } \rmx\in[0, X).  
\end{displaymath}
Since $\sfy$ is monotone, it is differentiable $\LL^1$-a.e. in $[0, X)$ and its derivative $\sfy'$ belongs to $L^1(0, X')$ for every $X' < X$. We define $\vartheta: [0, X) \to [0, +\infty)$, 
\begin{displaymath}
\vartheta(\mathrm x):= \int_0^\rmx{\sfy'(r) \ dr}.   
\end{displaymath}
The chain rule for absolutely continuous functions (see e.g. [\cite{leoni2009first}, Thm. 3.44]) applied to $\sfx\circ \sfy$ yields $\sfy'(r) > 0$ for $\LL^1$-a.e. $r\in (0, X)$. So it holds that
\begin{displaymath}
0 \ < \ \vartheta(\rmx_2) - \vartheta(\rmx_1) \ \leq \ \sfy(\rmx_2) - \sfy(\rmx_1) \quad \text{for all } 0\leq \rmx_1 < \rmx_2 < X, 
\end{displaymath}
and the map $\sfz: [0, T_\star(u)) \to [0, +\infty)$, defined as $\sfz:= \vartheta \circ \sfx$, is strictly increasing and $1$-Lipschitz, i.e.
\begin{displaymath}
0 \ < \ \sfz(t_2) - \sfz(t_1) \ \leq \ t_2 - t_1 \quad \text{for all } 0\leq t_1 < t_2 < T_\star(u). 
\end{displaymath}

The chain rule for absolutely continuous functions cannot be directly applied to $\sfy\circ \sfx$ since we do not know whether $\sfy$ is absolutely continuous or not, but imitating the proof of [\cite{leoni2009first}, Thm. 3.44], we obtain 
\begin{displaymath}
\sfy'(\sfx(t)) \sfx'(t) \ = \ 1 \quad \text{ a.e. in } \Omega.
\end{displaymath}
We used (\ref{eq: g u = u'}). By the chain rule, now applied to $\vartheta \circ \sfx$, it follows that
\begin{equation}\label{eq: sfz'}
\sfz'(t)  =  1 \quad\text{ a.e. in } \Omega, \quad \sfz'(t)  =  0 \quad\text{ a.e. in } [0, T_\star(u))\setminus \Omega. 
\end{equation}
Let
\begin{displaymath}
\theta \ := \ \lim_{\rmx\uparrow X} \vartheta(\rmx) \ = \ \int_0^X{\sfy'(r) \ dr} \in (0, +\infty].
\end{displaymath}
The map $\sfz$ has a strictly increasing, continuous inverse $\sft: [0, \theta) \to [0, T_\star(u))$. 

\medskip
We define $w: [0, \theta) \to \SS, \ w:= u\circ \sft$. It holds that
\begin{eqnarray*}
d(w(s), w(t)) \ \leq \ \int_{\sft(s)}^{\sft(t)}{|u'|(r) \ dr} \ = \ \sfx(\sft(t)) - \sfx(\sft(s)) 
\end{eqnarray*} 
for all $0\leq s \leq t < \theta$. Obviously, $\sfx\circ \sft$ is the inverse map of $\vartheta$. Since $\vartheta$ is locally absolutely continuous with $\vartheta'(r) = y'(r) > 0$ a.e. in $(0, X)$, its inverse $\sfx\circ \sft$ is locally absolutely continuous. By change of variables (see e.g. [\cite{leoni2009first}, Thm. 3.54]), we obtain  
\begin{displaymath}
\sfx(\sft(t)) - \sfx(\sft(s)) \ = \ \int_s^t{|u'|(\sft(r)) \sft'(r) \ dr} \quad 0\leq s \leq t < \theta. 
\end{displaymath}
It follows that $w\in AC_{\text{loc}}([0, \theta); \SS)$, i.e. the metric derivative 
\begin{displaymath}
|w'|(t):= \mathop{\lim}_{s\to t} \frac{d(w(s),w(t))}{|s-t|}
\end{displaymath} 
exists for $\LL^1$-a.e. $t\in(0, \theta)$, the function $t\mapsto |w'|(t)$ belongs to $L^1_{\text{loc}}(0, \theta)$ and 
\begin{displaymath}
d(w(s),w(t)) \leq \int^{t}_{s}{|w'|(r) \ dr}  \quad \text{for all } 0 \leq s\leq t < \theta;  
\end{displaymath}
moreover, it holds that 
\begin{equation}\label{eq: w' ch 4} 
|w'|(r) \leq |u'|(\sft(r))\sft'(r) \quad\text{a.e. in } (0, \theta)
\end{equation}
(cf. Definition \ref{def: lac curve}, [\cite{AmbrosioGigliSavare05}, Def. 1.1.1 and Thm. 1.1.2]). Applying the chain rule for absolutely continuous functions to $\sfz\circ \sft$, we obtain by (\ref{eq: sfz'}) that
\begin{equation}\label{eq: t' ch 4}
\sft'(r) \ = \ 1 \quad \text{a.e. in } \sfz(\Omega). 
\end{equation}
We note that the map
\begin{displaymath}
[0, \theta)\ni s \mapsto \int_0^{\sft(s)}{g^2(u(r)) \ dr}
\end{displaymath}
is strictly increasing and applying the chain rule for absolutely continuous functions, we obtain 
\begin{equation}\label{eq: ineq 1}
\int_{\sft(s_1)}^{\sft(s_2)}{g^2(u(r)) \ dr} \ \geq \ \int_{s_1}^{s_2}{g^2(u(\sft(r))) \sft'(r) \ dr}.
\end{equation}
Similarly, 
\begin{equation}\label{eq: ineq 2}
\int_{\sft(s_1)}^{\sft(s_2)}{|u'|^2(r) \ dr} \ \geq \ \int_{s_1}^{s_2}{|u'|^2(\sft(r)) \sft'(r) \ dr}.
\end{equation}
The curve $u$ satisfies the energy dissipation inequality (\ref{eq: EDI'}). Hence, combining (\ref{eq: w' ch 4}) - (\ref{eq: ineq 2}), we obtain
\begin{displaymath}
\phi(w(s)) - \phi(w(t)) \ \geq \ \frac{1}{2}\int_s^t{g^2(w(r)) \ dr} \ + \ \frac{1}{2}\int_s^t{|w'|^2(r) \ dr}
\end{displaymath}
for all $0\leq s \leq t < \theta$. 

\medskip
If $\theta < +\infty$ and $T_\star(u) = +\infty$, it holds that $w([0, \theta)) = \RR[u]$ and for every $T \in (0, \theta)$, the map $w_T: [0, +\infty) \to \SS$, 
\begin{displaymath}
w_T(s):= \begin{cases}
w(s) &\text{ if } 0\leq s \leq T \\
u(s - T + \sft(T)) &\text{ if } s > T
\end{cases}
\end{displaymath}
belongs to $\UU(\phi, g)$ (cf. the proof of Theorem \ref{thm: gf as gls}, \ref{itm: 2}). Since $\UU(\phi, g)$ satisfies \ref{itm: 4}, it follows that the limit $\lim_{t\uparrow +\infty} u(t) =: u_\star\in\SS$ exists, and $\bar{w}\in\UU(\phi, g)$, where
\begin{displaymath}
\bar{w}(t):= \begin{cases}
w(t) &\text{ if } 0\leq t < \theta \\
u_\star &\text{ if } t\geq \theta 
\end{cases}
\end{displaymath}
Moreover, it holds that $\bar{w}(\sfz(t)) = u(t)$ for all $t\in[0, +\infty)$. Hence, $u\succ \bar{w}$ which implies $u = \bar{w}$ since $u$ is minimal. We obtain \begin{equation}\label{eq: sfz = t}
\sfz(t) \ = \ t \quad \text{for all } t\in[0, T_\star(u))
\end{equation}
as $u$ is injective in $[0, T_\star(u))$. 

\medskip
If $T_\star(u) < +\infty$, then $\theta < +\infty$, and it is not difficult to see that $\bar{w}$ defined as above belongs to $\UU(\phi, g)$. Extending $\sfz$ by the constant value $\theta$, we again obtain $u\succ \bar{w}$, and thus (\ref{eq: sfz = t}). 

\medskip 
If $\theta = +\infty$ and $T_\star(u) = +\infty$, then $w\in\UU(\phi, g)$ and $u\succ w$, from which (\ref{eq: sfz = t}) follows. 

\medskip
So in any case, (\ref{eq: sfz = t}) holds. Taking into account (\ref{eq: sfz'}), we may conclude that
\begin{displaymath}
\LL^1([0, T_\star(u))\setminus \Omega) = 0. 
\end{displaymath}
This means that $u$ satisfies (\ref{eq: prop ii}). The proof is complete.    
\end{proof}

\paragraph{The strict monotonicity of $\phi$ along minimal solutions and (\ref{eq: prop ii})} 
Every minimal solution $u$ is injective in $[0, T_\star(u))$ due to Theorem \ref{thm: minimal solution}, \ref{itm: T2}. The functional $\phi$ is injective on $\RR[u]$ for every $u\in\UU(\phi, g)$ (Remark \ref{rem: injective}). It follows that $\phi$ is strictly decreasing along minimal solutions, i.e. 
\begin{equation}\label{eq: strict monotonicity}
\phi(u(t)) \ < \ \phi(u(s)) \quad \text{ for all } 0\leq s < t < T_\star(u), \quad u\in\UU_{\text{min}}(\phi, g), 
\end{equation}
where $\UU_{\text{min}}(\phi, g)$ denotes the collection of all the minimal solutions in $\UU(\phi, g)$. 

We note that 
(\ref{eq: strict monotonicity}) is not sufficient to conclude that a solution is minimal. In [\cite{sf2017}, Appendix A], we give an example of a one-dimensional gradient flow to a function whose derivative has a Cantor-like $0$ level set $K\subset \mathbb{R}$, and we construct a solution parametrized by a positive finite Cantor measure concentrated on $K$; this solution satisfies (\ref{eq: strict monotonicity}) but does not satisfy (\ref{eq: prop ii}) and is not minimal. The example illustrates that condition (\ref{eq: prop ii}) is stronger than (\ref{eq: strict monotonicity}) and that the strict monotonicity of the functional along a solution curve $u\in\UU(\phi, g)$ does not guarantee that $u\in\UU_{\text{min}}(\phi, g)$.

\medskip\noindent
\paragraph{Acknowledgements}

I would like to thank Giuseppe Savar\'e for stimulating discussions on this topic.

I gratefully acknowledge support from the Erwin Schr\"odinger International Institute for Mathematics and Physics (Vienna) during my participation in the programme ``Nonlinear Flows''.

I gratefully acknowledge support from the TopMath program, Technische Universit\"at M\"unchen.

\bibliographystyle{siam}
\bibliography{bibmaster}
\end{document}